\documentclass[10pt,a4paper,leqno]{amsart}
\usepackage{amssymb,amsmath,a4wide}
\usepackage{color}
\usepackage[pagebackref=true,colorlinks=true]{hyperref}

\usepackage{todonotes}
\newcommand*{\mailto}[1]{\href{mailto:#1}{\nolinkurl{#1}}}

\definecolor{darkgreen}{rgb}{0.5,0.25,0}
\definecolor{darkblue}{rgb}{0,0,1}
\definecolor{answerblue}{rgb}{0,0,0.75}

\hypersetup{colorlinks,breaklinks,
            linkcolor=darkblue,urlcolor=darkblue,
            anchorcolor=darkblue,citecolor=darkblue}

\newtheorem{theorem}{Theorem}[section]
\newtheorem{lemma}[theorem]{Lemma}
\newtheorem{proposition}[theorem]{Proposition}

\theoremstyle{definition}

\theoremstyle{remark}
\newtheorem{remark}[theorem]{Remark}

\numberwithin{equation}{section}

\newcommand{\abs}[1]{\left|#1\right|}

\def\D{\Delta}

\def\norm#1{\left\|#1\right\|}

\def \f12{\frac{1}{2}}

\newcommand{\eps}{\varepsilon}
\newcommand{\seq}[1]{\left\{#1\right\}}
\newcommand{\R}{\mathbb{R}}
\newcommand{\Z}{\mathbb{Z}}

\newcommand{\Do}{\R\times\R_+}
\newcommand{\loc}{\mathrm{loc}}

\newcommand{\mA}{\mathcal{A}}
\newcommand{\mD}{\mathcal{D}}
\newcommand{\mF}{\mathcal{F}}
\newcommand{\mG}{\mathcal{G}}
\newcommand{\tmA}{\tilde{\mathcal{A}}}
\newcommand{\tmD}{\tilde{\mathcal{D}}}
\newcommand{\tmE}{\tilde{\mathcal{E}}}
\newcommand{\tI}{\tilde{I}}
\newcommand{\tGamma}{\tilde{\Gamma}}
\newcommand{\mE}{\mathcal{E}}

\newcommand{\mU}{\mathcal{U}}

\newcommand{\mY}{\mathcal{Y}}
\newcommand{\mZ}{\mathcal{Z}}
\newcommand{\dsum}{ \mathop{\sum \sum}}
\newcommand{\sumj}{\sum_{j \in \Omega_n}}
\newcommand{\sumjp}{\sum_{j \in \Omega'_n}}

\newcommand{\suml}{\sum_{\ell \in \Omega_n}}
\newcommand{\sumn}{\sum_{n=0}^{N-1}}
\newcommand{\sumnN}{\sum_{n=0}^N}
\newcommand{\sumnNt}{\sum_{n=0}^{N-\theta}}
\newcommand{\mR}{\mathcal{R}}
\newcommand{\tmR}{\tilde{\mathcal{R}}}
\newcommand{\mS}{\mathcal{S}}
\newcommand{\tmS}{\tilde{\mathcal{S}}}

\newcommand{\Dx}{\Delta x}
\newcommand{\Dt}{\Delta t}

\newcommand{\ud}{u^{\Delta}}
\newcommand{\kd}{k^{\Delta}}
\newcommand{\jpne}{j+n=\mathrm{even}}
\newcommand{\jpno}{j+n=\mathrm{odd}}
\newcommand{\even}{\mathrm{even}}
\newcommand{\odd}{\mathrm{odd}}
\newcommand{\Dtau}{\Delta^\tau}

\newcommand{\pt}{\partial_t}
\newcommand{\px}{\partial_x}

\newcommand{\tA}{\tilde{A}}
\newcommand{\tB}{\tilde{B}}
\newcommand{\tC}{\tilde{C}}
\newcommand{\tD}{\tilde{D}}
\newcommand{\tE}{\tilde{E}}

\newcommand{\ue}{u_{\eps}}

\newcommand{\pu}{\partial_u}
\newcommand{\pk}{\partial_k}

\newcommand{\pxx}{\partial^2_{xx}}
\newcommand{\pyy}{\partial^2_{yy}}

\newcommand{\puu}{\partial^2_{uu}}

\newcommand{\supp}{\mathrm{supp}}

\newcommand{\Linf}{L^{\infty}}
\newcommand{\dx}{\, dx}

\newcommand{\dt}{\, dt}

\DeclareMathOperator{\dint}{\int\!\!\!\!\!\int\!\!\!}

{%

\begin{enumerate}}%
{\end{enumerate}}

{%

\begin{enumerate}}%
{\end{enumerate}}

\begin{document}

\title[Discontinuous flux and quantitative compactness estimates]
{Compactness estimates for difference schemes 
for conservation laws with discontinuous flux}

\author[Karlsen]{Kenneth H. Karlsen}
\address[Kenneth H. Karlsen]
{\newline Department of mathematics, University of Oslo
\newline P.O. Box 1053,  Blindern, NO--0316 Oslo, Norway}
\email[]{\mailto{kennethk@math.uio.no}}

\author[Towers]{John D. Towers}
\address[John D. Towers]{\newline
MiraCosta College\newline
3333 Manchester Avenue \newline
Cardiff-by-the-Sea, CA 92007-1516, USA}
\email{\mailto{john.towers@cox.net}}

\date{\today}



\begin{abstract}
We establish quantitative compactness estimates for finite difference schemes 
used to solve nonlinear conservation laws. These equations 
involve a flux function $f(k(x,t),u)$, where the coefficient $k(x,t)$ is 
$BV$-regular and may exhibit discontinuities along curves in the $(x,t)$ plane. 
Our approach, which is technically elementary, relies on 
a discrete interaction estimate and one entropy function. 
While the details are specifically outlined 
for the Lax-Friedrichs scheme, the same framework can 
be applied to other difference schemes. Notably, our compactness 
estimates are new even in the homogeneous case ($k\equiv 1$).
\end{abstract}

\maketitle


\section{Introduction} \label{Introduction}
The main part of this paper investigates a finite difference
algorithm as it applies to the Cauchy problem for scalar
conservation laws with the form
\begin{equation}\label{cauchy}
	u_t+f(k(x,t),u)_x=0, \qquad u(x,0)=u_0(x),
\end{equation}
where $(x,t)\in \Do$; $u(x,t)$ is the scalar unknown
function; and $u_0(x),k(x,t),f(k,u)$ are given functions to be detailed later.
Here it suffices to say that for the compactness estimates we
need $k(x,t)\in BV(\R\times\R_+)$, $u\mapsto f(k(x,t),u)$
genuinely nonlinear, and $u_0(x)$ bounded 
(see Section~\ref{sec:LxFr} for the complete list of assumptions).

The special feature of \eqref{cauchy} 
is the nonlinear flux function $f(k(x,t),u)$ that depends
explicitly on the spatial and temporal variables through
a discontinuous coefficient $k(x,t)$. 
Conservation laws with discontinuous flux functions are encountered 
in various applications. For example, in oil reservoirs, rock permeability 
may vary significantly in different locations, resulting in a discontinuous flux function 
(see, e.g., \cite{Gimse:1992ma}). Similarly, in traffic flow, abrupt 
changes in vehicle density may occur at bottlenecks, which leads to a discontinuous flux 
(see, e.g., \cite{Burger:2009fk}).  

The study of conservation laws with discontinuous flux has been 
a heavily investigated area for the past three decades. 
This is partly due to its links to various applications, but also 
because it possesses several non-trivial mathematical properties. 
These properties include the existence of several $L^1$ stable semigroups, 
which are based on different entropy conditions, as well as the lack of uniform bounds 
on the total variation (i.e., $BV$ estimates). 
As a result, it constitutes a non-trivial generalization of scalar conservation 
laws \cite{Kruzkov:1970kx}. While a comprehensive review 
of the extensive literature is beyond the scope of this paper, 
we provide a few select references for interested readers and direct 
them to the reference lists in these papers 
\cite{Adimurthi:2005kx, Andreianov:2010fk, Andreianov:2023aa, 
bressan2019vanishing, Burger:2008sf, Klingenberg:1995ly, 
Karlsen:2003lz, Karlsen:2004sn}.
 
Proving the existence of solutions for conservation laws is 
associated with establishing strong a priori estimates for  
approximate solutions. Classically, this involves bounding 
the total variation of the approximate solutions, independent 
of the approximation parameter.
However, bounding the total variation when the flux is 
discontinuous is in general impossible. To address this issue, several 
alternative convergence approaches have been applied over the years, including
singular mapping techniques as well as compensated compactness 
and other advanced weak convergence methods 
\cite{Erceg:2023aa, Holden:2009nx, Karlsen:2003lz, Karlsen:2004sn, 
Klingenberg:1995ly, Panov-H-measure:10} (this is just a few examples).
Specifically, the weak convergence methods 
(see, e.g., \cite{Erceg:2023aa,Panov-H-measure:10}) are profound 
and involve a significant amount of functional analysis, making 
them technically challenging to comprehend.

The aim of this paper is to derive quantitative $L^1$ translation estimates 
that can be utilized for various applications, such as demonstrating 
the convergence of finite difference approximations. 
Let us consider a time step $\Dt$ and a grid 
size parameter $\Dx$, which collectively form a 
parameter pair $\Delta = (\Dx, \Dt)$. We use the 
notation $\seq{u^{\Delta}}_{\Delta > 0}$ to represent 
the approximate solutions. The translation estimates, which 
maintain uniformity across all $\Delta$, are defined as follows:
\begin{equation}\label{eq:comp-est-intro1}
	\int_0^{T-\tau}\int_{\R}
	\abs{u^\D(x+h,t+\tau)-u^\D(x,t)}\chi(x)
	\dx \dt \lesssim_{T,\chi} \tau^{\mu_t}+h^{\mu_x}.
\end{equation}
This estimate holds true for any temporal ($\tau > 0$) 
and spatial ($h > 0$) translations. 
Here, $\mu_t,\mu_x$ are parameters in the interval $(0,1)$ that 
quantify the ``degree of compactness", where $\mu_t=\mu_x=1$ 
corresponds to uniformly bounded total variation. 

In \eqref{eq:comp-est-intro1}, $\chi$ is a weight function that can 
be employed to control the growth of solutions as they approach infinity in $x$. 
An example of a weight function is $\chi(x)=(1+\abs{x}^2)^{-N}$, where $N>1/2$, 
see \eqref{weight_properties} and also \cite{Karlsen:2023aa}. 
If $\supp \left(u^{\D}\right)$ is contained within a compact set 
$[0,T]\times [-R,R]\subset \R^2$, independent of $\D$, we 
may set $\chi\equiv 1$, which is what we do 
for the rest of the introduction!

We refer to translation estimates like \eqref{eq:comp-est-intro1} 
as \textit{quantitative compactness estimates}. 
They can be used to derive convergence results via 
the well-known Kolmogorov--Riesz--Fr\'echet
characterization of precompact subsets of $L^1$ in terms of the uniform 
continuity of translations in $L^1$, see, e.g., \cite[Theorem 4.26]{Brezis:2010aa}. 
Beyond implying convergence, quantitative compactness estimates 
can be used to derive continuous dependence and 
error estimates \cite{Kuznetsov:1976ys,Bouchut:1998ys}. 

Our approach is technically elementary and relies 
on discrete interaction estimates and 
the existence of one uniformly convex entropy. We draw inspiration from previous 
work of Golse and Perthame \cite{Golse:2013aa}, who developed a 
quantitative compensated compactness framework for 
establishing Besov space regularity of solutions to 
homogenous conservation laws. 
We adapt their approach and apply it to establish quantitative 
compactness estimates for sequences of approximate solutions. 
A related approach has recently been employed in the study 
of vanishing viscosity approximations of stochastic 
conservation laws, as described in \cite{Karlsen:2023aa}.
While the details are specifically outlined 
for the Lax-Friedrichs scheme, the same framework can 
be applied to other difference schemes as well. 
The obtained estimates are new also in the homogenous 
case $k\equiv 1$ (scalar conservation law).

To clarify further, let us consider the Lax-Friedrichs scheme 
for the Cauchy problem \eqref{cauchy}, which can be 
represented by the following equation \cite{Karlsen:2004sn}:
\begin{equation}\label{eq:LxF-intro1}
	U_j^{n+1}=\frac12\left(U_{j-1}^n+U_{j+1}^n\right)- \frac{\lambda}{2}
	\left(f_{j+1}^n - f_{j-1}^n\right),
\end{equation}
Here, $U_j^n$ approximates the exact solution $u$ at the grid 
point $(j\Dx,n\Dt)$, where $n$ and $j$ are 
integers, $\lambda = \Dt/\Dx$, and 
$f_j^n=f\left(k_{j}^n,U_{j}^n\right)$. The temporal 
and spatial discretization parameters are linked through 
a CFL condition \eqref{ass:CFLkappa}, so that $\Dt\sim \Dx$.

Assuming reasonable conditions on $f$, $k\in BV$, and $u_0\in L^\infty$, we can 
establish the following two a priori estimates (see next section):
\begin{equation}\label{eq:apriori-intro1}
	\sup_{n,j} \abs{U_j^n}\lesssim 1, 
	\quad  \Dx \sum_n\sum_j
	\left(U_{j+1}^n - U_{j-1}^n\right)^2\lesssim 1.
\end{equation}
These estimates are the only bounds that remain robust with 
respect to the regularity of the coefficient $k(x,t)$. If $k\equiv 1$, the 
approximations are of bounded total variation \cite{Holden:2015aa}. 
The second bound is referred to as a dissipation (or entropy stability) 
estimate, which is known to hold for many numerical schemes 
for conservation laws \cite{Eymard:2000fr,Tadmor2003}. 
These estimates are sometimes referred to as weak $H^1$ 
estimates since they suggest that 
$\iint \abs{\px \ud}^2\dx\dt\lesssim 1/\abs{\D}$. 
The estimate for the Lax-Friedrichs scheme \eqref{eq:LxF-intro1} 
with variable and discontinuous $k$ was derived in \cite{Karlsen:2004sn}. 
In this paper, we present a slightly 
generalized variant of the estimate.

Given the \textit{genuinely nonlinear} flux $f(k,u)$, in the sense 
of \eqref{f_S_assumptions}, consider an 
entropy/entropy flux pair $\bigl(S(k,u),Q(k,u)\bigr)$. 
Set $S_j^n = S\left(k_j^n,U_j^n\right)$ and 
$Q_j^n = Q\left(k_j^n,U_j^n\right)$. 
The Lax-Friedrichs scheme satisfies the following 
entropy balance (to be derived later): 
\begin{equation}\label{eq:LxF-entropy-intro1}
	S_j^{n+1}-\frac12\left(S_{j-1}^n+S_{j+1}^n\right)
	+ \frac{\lambda}{2} \left(Q_{j+1}^n-Q_{j-1}^n \right) 
	=  \Psi_j^n,
\end{equation}
where, by \eqref{eq:apriori-intro1} and $k\in BV$, we have
$\Dx \sum_n\sum_j \abs{\Psi_j^n} \lesssim 1$. Note 
that the entropy production $\Psi_j^n$ is a signed 
measure (even if $S=S(u)$ is convex), which is 
distinct from the homogenous conservation 
law case where it is negative. 

Let $\nu$ be an arbitrary nonnegative integer and introduce 
the ``spatial difference" quantities
\begin{align*}
	&A_j^n =U_{j+2\nu}^n-U_j^n, 
	\quad B_j^n =f_{j+2\nu}^n-f_j^n,
	\\ & 
	 D_j^n =S_{j+2\nu}^n-S_j^n, 
	\quad E_j^n =Q_{j+2\nu}^n-Q_j^n.
\end{align*}
The appearance of $2\nu$, as opposed to simply $\nu$, in the above
equations is due to that we use a staggered mesh, where the various 
grid functions are only defined at grid points where $j+n$ is an even integer.
Then it is straightforward to deduce from \eqref{eq:LxF-intro1} 
and \eqref{eq:LxF-entropy-intro1} that the following 
$2\times 2$ system of finite difference equations hold:
\begin{equation}\label{eq:system-intro}
	\begin{split}
		 & A_j^{n+1} - \f12 \left(A_{j-1}^n + A_{j+1}^n \right) 
 		+{\lambda \over 2} \left(B_{j+1}^n - B_{j-1}^n \right) = C_{A,j}^n, 
		\\ & 
		D_j^{n+1} - \f12 \left(D_{j-1}^n + D_{j+1}^n \right) 
		+{\lambda \over 2} \left(E_{j+1}^n - E_{j-1}^n \right) 
		= C_{D,j}^n,
	\end{split}
\end{equation}
where $C_{A,j}^n\equiv 0$ and $C_{D,j}^n= \Psi_{j+2\nu}^n- \Psi_j^n$. 
We are also going to need the following more regular quantities 
obtained by applying ``inverse-difference" operators to $A_j^n$, $D_j^n$:
\begin{align*}
	\mA_\ell^n =  \Dx  \sum_{j\le \ell} A_j^n, 
	\quad 
	\mD_j^n = \Dx  \sum_{\ell\ge j} D_\ell^n.
\end{align*}
Indeed, one can prove that 
$$
\sup_{n,j}\abs{\mA_{\ell}^n} \lesssim 2\nu \Dx, 
\quad 
\sup_{n,j}\abs{\mD_j^n} \lesssim 2\nu \Dx,
$$ 
recalling that $\nu$ is a nonnegative integer.

Next, we introduce the \textit{interaction functional}
\begin{equation*}
	I^n = \Dx^2 \dsum_{j \le \ell} A_j^n D_\ell^n,
\end{equation*}
which is a measure of future potential interaction at time level $n$ 
of the finite difference solutions. This is a discrete version of a functional 
referred to as the Varadhan functional by Tartar \cite[p. 182]{Tartar:book08} and 
Golse \& Perthame \cite{Golse:2013aa}.

In this paper, we establish the following discrete interaction identity 
for the $2\times 2$ system \eqref{eq:system-intro}:
\begin{equation}\label{eq:interaction-intro1}
	\f12 \Dt  \Dx  \sum_n \sum_j  \left(A_j^n E_j^n -D_j^n B_j^n\right)
	= - \Dx  \sum_n \sum_j  C_{D,j}^n \mA_j^{n+1}+\mE, 
\end{equation}
To obtain an evolution difference equation satisfied by 
the interaction potential $I^n$, we compute the temporal 
difference $\frac{I^{n+1}-I^n}{\Dt}$. 
This identity is then multiplied by $\Dt$ and summed over $n$, resulting in 
the equation given by \eqref{eq:interaction-intro1}. 
The derivation of \eqref{eq:interaction-intro1} employs a discrete chain 
rule, as well as the difference equations for $A_j^n$ and $D_j^n$. 
A more complicated form of \eqref{eq:interaction-intro1} holds when there is a 
weight function $\chi$ present. The interaction identity \eqref{eq:interaction-intro1} 
can be interpreted as a discrete form of the interaction 
identity (7) introduced in \cite{Golse:2013aa}, as well 
as the stochastic interaction identity 
(3.12) in \cite{Karlsen:2023aa} (with $\sigma =0$). 
It should be noted that the term $\mE$ appearing on the right-hand 
side of \eqref{lxf_interaction} is solely a consequence of the discretization 
process and does not have a corresponding counterpart in the interaction 
identities of \cite{Golse:2013aa,Karlsen:2023aa}. 
For the detailed form of $\mE$, see Lemma \ref{lemma_lxf_interaction}. 

The terms on the right-hand side of \eqref{eq:interaction-intro1} can 
all be bounded by a $\D$-independent constant $C$ times $2\nu \Dx$. 
The next step is to convert \eqref{eq:interaction-intro1} into a quantitative 
compactness estimate, which requires exploiting the genuine nonlinearity 
of $f(k,u)$ in $u$. The precise assumption can be found 
in \eqref{f_S_assumptions}, and a corresponding assumption 
for homogeneous conservation laws 
can be found in \cite[Section 5]{Golse:2013aa}. 
A relevant special case occurs if $u\mapsto f(k,u)$ is 
uniformly convex and $S=f$. Then 
$$
A_j^n E_j^n -D_j^n B_j^n\gtrsim \abs{U_{j+2\nu}^n-U_j^n}^4,
$$ 
and therefore 
\begin{equation}\label{eq:spatial-discrete}
	\Dt  \Dx  \sum_n \sum_j  \abs{U_{j+2\nu}^n-U_j^n}^4
	\lesssim 2 \nu \Dx,
\end{equation}
A similar estimate can be deduced for the temporal differences. 
Along with the compact support assumption (or if we use a weight function) 
and H\"older's inequality, this translates into the $L^1$ 
translation estimate given by \eqref{eq:comp-est-intro1} 
with $h=\tau=2\nu \Dx$ and $\mu_t=\mu_x=1/4$, which is the 
main result of this paper. 

In summary, the above outline provides an overview of the 
quantitative compactness approach. To delve deeper, we will 
present detailed proofs incorporating a weight function in the following sections.
We demonstrate the effectiveness of the quantitative 
compactness approach on the Lax-Friedrichs scheme. However, with 
slight modifications to the discrete interaction identity given 
by \eqref{eq:interaction-intro1}, the same approach can be applied 
to other classical schemes, provided they are 
uniformly bounded in $L^\infty$ and satisfy a 
weak $H^1$ (dissipation) estimate, that is to say, they satisfy 
modified forms of the a priori estimates 
described in \eqref{eq:apriori-intro1}. 
This allows for elementary convergence proofs, especially 
in cases where obtaining total variation estimates is challenging.

Our results provide an existence theorem for 
the Cauchy problem \eqref{cauchy}.  
The uniqueness question for conservation laws with 
discontinuous flux is another problem entirely, even in
the most basic setting where the flux has a single 
spatial discontinuity and no temporal
discontinuity (the so-called two-flux problem). 
For example, it turns out that the two-flux problem
generally has infinitely many definitions of entropy solution, each 
one of which generates its own distinct $L^1$ contraction 
semigroup \cite{Adimurthi:2005kx}. In \cite{Karlsen:2004sn} we 
used results from \cite{Karlsen:2003lz} to prove a uniqueness 
result applicable to the Lax-Friedrichs scheme of this paper 
in the special case where $k$ is piecewise Lipschitz continuous,
meaning that all of the discontinuities of $k$ occur along 
Lipschitz continuous curves in the $(x,t)$-plane. 
This uniqueness result was proven under the additional 
assumption that all of the flux discontinuities satisfy a certain
``crossing condition''. At least for the simple two-flux version 
of the problem \eqref{cauchy}, the solution generated by
the Lax-Friedrichs scheme of this paper corresponds to 
the so-called vanishing viscosity solution \cite{Andreianov:2010fk}. 
In the absence of the crossing condition it is not known 
whether the (subsequential) limit of the Lax-Friedrichs scheme
is the vanishing viscosity solution.  
Finally, given \eqref{eq:spatial-discrete}, it is worth 
noting that the solution $u$ derived as the limit 
of the Lax-Friedrichs scheme,
exhibits Besov space regularity as given by
$$
u \in B_{ \infty,\loc}^{1/4,4}(\R^+ \times \R).
$$
This regularity aligns with the known regularization effect established 
in \cite[Theorem 5.1]{Golse:2013aa} for homogeneous 
equations with a single convex entropy.

\medskip

The structure of this paper is as follows: Section \ref{sec:LxFr} presents 
the assumptions related to the data of the problem and precisely 
defines the Lax-Friedrich scheme. This section also 
introduces our main result. In Section \ref{sec:prelim-est}, we 
establish preliminary $L^\infty$ and weak $H^1$ estimates. 
Section \ref{sec:spatial_est} proves the spatial translation estimate, 
while Section \ref{sec:temporal_est} details the temporal estimate.
By bringing together the spatial and temporal estimates, we provide 
the proof of our main result.

\section{Lax-Friedrichs scheme and main result}
\label{sec:LxFr}

We begin by listing
some assumptions on $u_0,k,f$ which
will be needed.

\noindent Regarding the initial function we assume
\begin{equation}\label{ass:init_data_en}
	u_0 \in \Linf(\R), \qquad a \leq u_0(x)
	\leq b \quad \text{for a.e.~$x\in \R$}.
\end{equation}

\noindent For the discontinuous coefficient $k:\Do\to \R$ we assume that
\begin{equation}\label{ass:k}
	k \in \Linf(\Do)\cap BV(\R\times\R_+), \quad \alpha \leq k(x,t)
	\leq \beta \quad \text{for a.e.~$(x,t)\in \Do$}.
\end{equation}

\noindent Regarding the flux function 
$f:[\alpha,\beta]\times [a,b]\to \R$ we assume that
\begin{equation}\label{ass:f_en}
	f \in C^1([\alpha,\beta]\times [a,b]).
\end{equation}
We need also an assumption on $f$ that guarantees
that the Lax-Friedrichs approximations
stay uniformly bounded. For example, we can require
\begin{equation}\label{ass:f_to}
   \text{$f(k,a)=f(k,b)=0$ for all $k\in [\alpha,\beta]$},
\end{equation}
which in fact implies that the interval $[a,b]$ becomes
an invariant region.

We use the notation $\pk G$ and $\pu G$ to denote the 
first order partial derivatives of $G(\cdot,\cdot)$ 
with respect to the first and second variables.
Letting $\mU := [\alpha,\beta] \times [a,b]$, we will 
use the following abbreviations:
\begin{equation*}
	\norm{\partial_u G}= \max_{(k,u) \in \mU}\abs{\partial_u G(k,u)},
	\quad 
	\norm{\partial_k G}= \max_{(k,u) \in \mU}\abs{\partial_k G(k,u)} 
	\quad 
	\textrm{for $G = f, S, Q$}.
\end{equation*}
We are given functions $S$ and $Q$ (referred to earlier) 
that are assumed to form an entropy/entropy 
flux pair $(S(k,u),Q(k,u))$, meaning that
\begin{equation}\label{def_ent_pair}
	\pu Q(k,u)= \pu S(k,u) \pu f(k,u).
\end{equation}
We assume that
\begin{equation}\label{ass:S_Q_1}
	S \in C^2([\alpha,\beta]\times [a,b])
	\quad 
	(\text{and thus $Q \in C^1([\alpha,\beta]\times [a,b])$}).
\end{equation}

We make the following \textit{genuine nonlinearity} 
assumptions about $f$ and $S$. For $k \in [\alpha,\beta]$,
\begin{equation}\label{f_S_assumptions}
	\begin{split}
		& \pu f(k,v) - \pu f(k,w) \ge C_f (v-w)^{p_f}, 
		\quad a \le w < v \le b, \quad p_f \ge 1, 
		\\ &\pu S(k,v) - \pu S(k,w) \ge C_S (v-w)^{p_{S}}, 
		\quad a \le w<v \le b, \quad p_S \ge 1, 
	\end{split}
\end{equation}
for some constants $C_f, C_S >0$.

Next we describe the Lax-Friedrichs scheme.
Let $\Dx>0$ and $\Dt>0$ denote the spatial
and temporal discretization parameters, which
are chosen so that they always obey the CFL condition
\begin{equation}
   \label{ass:CFLkappa}
   \lambda \|\pu f\|
   \leq 1-\kappa, \qquad
   \lambda = \frac{\Dt}{\Dx},
   \qquad \text{for some $\kappa\in (0,1)$.}
\end{equation}
Here $\kappa$ is a positive parameter which we can choose to
be very small so that the allowable time step is reduced only
negligibly.
We will work under the standing
assumption that the space step $\Dx$ and the time step $\Dt$
are comparable, i.e., there are constants $c_1,c_2>0$ such that
$c_1\le \frac{\Dt}{\Dx}\le c_2$. Therefore, when we declare that a 
constant $C$ is independent of $\Dx$ (or $\Dt$), it 
implies that $C$ is also independent of $\Delta = (\Dx, \Dt)$.

The time domain $[0,\infty)$ is discretized via $t^n = n\Dt$ for
$n\in\Z_+^0:=\{0,1,\ldots\}$ ($\Z_+:=\{1,2,\ldots\}$), resulting
in time strips $[t^n,t^{n+1})$. The spatial domain $\R$  is divided into
cells $[x_{j-1},x_{j+1})$ with centers at
the points $x_j =j \Dx$ for $j \in\Z$. Let
$\rho_j(x)$ be the characteristic function for the interval $[x_{j-1},x_{j+1})$
and $\rho_j^n$ the characteristic function for the rectangle
$[x_{j-1},x_{j+1}) \times [t^n,t^{n+1})$.

The finite difference scheme then generates, for each mesh
size $\Delta = (\Dx,\Dt)$, with $\Dx$ and  $\Dt$ taking
values in sequences tending to zero, a piecewise constant approximation
\begin{equation}\label{eq:ud_def}
   \ud(x,t)=\sum_{n\in \Z_+^0}\underset{\jpne}{\sum_{j\in \Z}}
   \rho_j^n(x,t)U_j^n,
\end{equation}
where the values $\left\{U_j^n:(j,n)\in\Z\times \Z_+^0, \jpne\right\}$
remain to be defined.

We define $\left\{U_j^0: j=\even\right\}$ by
\begin{equation}\label{initdata}
    U_j^0 = \frac{1}{2\Dx}\int_{x_{j-1}}^{x_{j+1}}u_0(x)\dx.
\end{equation}

\noindent
Given $\left\{U_j^n:\jpne\right\}$, we define next $\left\{U_j^{n+1}:\jpno\right\}$.
Let $(K,U)=(K,U)(x,t)$ denote a weak solution of the $2\times 2$ system
\begin{equation}
   \label{eq:system}
   K_t= 0, \qquad U_t + f(K,U)_x=0, \qquad (x,t)\in \Do,
\end{equation}
with Riemann initial data
$$
K(x,0)=
\begin{cases}
   k_{j-1}^n, & x<x_j, \\
   k_{j+1}^n, & x>x_j,
\end{cases}
\qquad
U(x,0)=
\begin{cases}
   U_{j-1}^n, & x<x_j, \\
   U_{j+1}^n, & x>x_j,
\end{cases}
$$
where the coefficient $k(x,t)$ has been
discretized via the piecewise constant approximation
\begin{equation}\label{k_init}
	\kd(x,t)= \sum_{n\in \Z_+^0}\underset{\jpne}{\sum_{j\in \Z}}
   	\rho_j^n(x,t)k_j^n,
	\qquad 
	k_j^n = k(x_j,\hat{t}_n).
\end{equation}
Recall that $k(x,t)$ is a $BV$ function. For every $t$, the 
function $k(t,\cdot)$ can be considered as a precise representative, 
being defined everywhere and normalized through right-continuity. 
Consequently, we are justified in setting 
$k_j^n = k(x_j,\hat{t}_n)$ in \eqref{k_init}, where 
$\hat{t}_n$ is an arbitrary point lying in the interval
$[t^n,t^{n+1})$ (for example, $\hat{t}_n=t^n$). We then define
\begin{equation*}
U_j^{n+1}=\frac{1}{2\Dx} \int_{x_{j-1}}^{x_{j+1}} U(x,\Dt)\dx.
\end{equation*}
Integrating the weak formulation of \eqref{eq:system}
over the control volume $[x_{j-1},x_{j+1})\times(0,\Dt)$ gives
\begin{align*}
   &\int_{x_{j-1}}^{x_{j+1}} U(x,\Dt)\dx
   = \int_{x_{j-1}}^{x_{j+1}} U(x,0)\dx
   \\ & \qquad
   - \int_0^{\Dt} \left(f\left(K(x_{j+1},t),U(x_{j+1},t)\right)
   -f\left(K(x_{j-1},t),U(x_{j-1},t)\right)\right)\dt.
\end{align*}

\noindent After a direct evaluation of the integrals for $\Dt$
small, we obtain the staggered
Lax-Friedrichs scheme
\begin{equation}\label{LxFr_scheme}
	U_j^{n+1}=\frac12\left(U_{j-1}^n+U_{j+1}^n\right)
	-\frac{\lambda}{2}\left(f\left(k_{j+1}^n,U_{j+1}^n\right)
	-f\left(k_{j-1}^n ,U_{j-1}^n\right)\right).
\end{equation}

Notice that in this paper we restrict our attention
to the sublattice
$$
\left\{(x_j,t_n): \jpne\right\},
$$
which means that $\left\{U_j^0: j=\even\right\}$,
$\left\{U_j^1: j=\odd\right\}$,
$\left\{U_j^2: j=\even\right\}$ etc.~ are calculated.
The following abbreviations can be used to shorten some of 
the expressions that will arise.
For fixed $n \in \{0, 1, \ldots, N\}$,
\begin{equation*}
	\Omega_n = \{j \in \Z: n+j 
	= \textrm{even}\}, \quad 
	\Omega'_n = \{j \in \Z: n+j = \textrm{odd}\}.
\end{equation*} 
Note that 
\begin{equation*}
	\textrm{$U_j^n, U_{j \pm 2}^n, \ldots$ 
	and $U_{j\pm 1}^{n+1}, U_{j \pm 1}^{n+1}, \ldots$ 
	are defined for $j \in \Omega_n$},
\end{equation*}
while
\begin{equation*}
	\textrm{$U_{j\pm 1}^{n}, U_{j \pm 3}^{n}, 
	\ldots$ and $U_{j}^{n+1}, U_{j \pm 2}^{n+1}, \ldots$
are defined for $j \in \Omega'_n$.}
\end{equation*}

We work with the so-called ``even'' sublattice 
described above in the interest of conceptual
simplicity. Note that when one applies the Lax-Friedrichs 
scheme on the standard lattice, what is generated is two 
uncoupled numerical solutions, one solution on the even sublattice, and
one solution on the odd sublattice. Our analysis for the even 
sublattice would then apply to each of those solutions separately.

The following theorem is our main result:

\begin{theorem}\label{thm:main}
Suppose $u_0$ satisfies \eqref{ass:init_data_en}, 
the discontinuous coefficient $k$ satisfies \eqref{ass:k}, the flux 
$f$  satisfies \eqref{ass:f_en}, \eqref{ass:f_to}, 
and the genuine nonlinearity assumption in \eqref{f_S_assumptions}. 
Suppose also that we are given an entropy/entropy 
flux pair $(S,Q)$ that satisfies \eqref{def_ent_pair}, \eqref{ass:S_Q_1}, 
and \eqref{f_S_assumptions}. Let the spatial and temporal discretization 
parameters $\D=(\Dx,\Dt)$ obey the CFL condition \eqref{ass:CFLkappa}.
Denote by $u^\D(x,t)$ the piecewise constant 
Lax-Friedrichs approximation defined 
by \eqref{eq:ud_def}, \eqref{initdata}, \eqref{k_init}, and 
\eqref{LxFr_scheme}.  Then the following 
quantitative compactness estimate holds:
\begin{equation}\label{eq:comp-est-main-thm}
	\int_0^{T-\tau}\int_{\R}
	\abs{u^\D(x+h,t+\tau)-u^\D(x,t)}\chi(x)
	\dx \dt \leq C\bigl(\tau^{\mu}+h^{\mu}\bigr),
	\quad \mu:= \frac{1}{p_f+p_{S}+2},
\end{equation}
for $h>0$ and $\tau \in (0,T)$. Here, $C=C_{T,\chi}$ is a 
constant independent of $\D$, and $\chi \in C^1(\R) \cap L^1(\R)$ is 
a weight function satisfying the following conditions for all $x\in \R$:
\begin{equation}\label{weight_properties}
	\begin{split}
		& \chi(x)>0, 
		\quad
		\abs{\chi'(x)} \le C_{\chi} \chi(x), 
		\quad \text{and}
		\\ & 
		\abs{\chi(x+z)-\chi(x)} 
		\lesssim \chi(x) \abs{z},
		\quad
		\sup_{\abs{x-y} \le R} 
		{\chi(x)\over \chi(y)} \lesssim_R 1.
	\end{split}
\end{equation}
If $u^\D(\cdot,t)$ is compactly supported, or 
more generally, if it is bounded in $L^1(\R)$, 
uniformly in $\D$ and $t\in [0,T]$, we 
can choose the weight function $\chi$ to be 
equal to one (i.e., $\chi\equiv 1$).
\end{theorem}

The validation of \eqref{eq:comp-est-main-thm} 
is directly derived from the results 
in Section \ref{sec:spatial_est} 
(Proposition \ref{prop_space_est})
and Section \ref{sec:temporal_est} 
(Proposition \ref{prop_time_est}).

\begin{remark}
Assuming that the entropy $u \mapsto S(k,u)$ 
in \eqref{f_S_assumptions} is uniformly convex,
\begin{equation}\label{S_convex}
	\puu S(k,u) \ge \gamma >0, 
	\quad \forall (k,u) \in [\alpha,\beta] \times [a,b],
\end{equation}
the exponent $\mu$ becomes$\frac{1}{p_f+3}$, where $p_f\ge1$ 
is given by \eqref{f_S_assumptions}. Suppose 
$f \in C^2([\alpha,\beta] \times [a,b])$ is uniformly 
convex in $u$: $\puu f(k,\xi) \ge \gamma >0$ for all 
$(k,\xi) \in [\alpha,\beta] \times [a,b]$. If we 
choose $S = f$ as the entropy, then $\mu=\frac14$, 
see also Remark \ref{remark_eta_eq_f}.
\end{remark}

\begin{remark}
In this paper, we employ estimates denoted 
as ``$a\lesssim b$", signifying that there exists a constant $C$ 
such that $a \leq C b$. Notably, $C$ may 
depend on the specific constants associated 
with the assumptions of the problem. However, $C$ 
does not depend on the grid parameters $\D$.
\end{remark}

We will conclude this section by outlining the key 
concepts underpinning the proof of the spatial part of 
Theorem \ref{thm:main}. Our discussion will be 
focused on a homogenous equation, augmented with 
artificial viscosity, to provide a clear outline of the underlying ideas. 

For any fixed $\eps>0$, let 
$\ue\in C^2$ satisfy the equations
\begin{equation}\label{eq:par-system}
	\begin{split}
		& \pt \ue +\px f(\ue)=\eps \pxx \ue, 
		\\ & 
		\pt \eta(\ue)+\px q(\ue)
		=\eps\pxx \eta(\ue)-\mu_\eps,
	\end{split}
\end{equation}
where $f,\eta\in C^2$ 
are (for example)  uniformly convex and
$$
\mu_\eps:=\eta''(\ue)
\eps \left(\px \ue\right)^2, \quad 
\int_0^\infty \int_{-\infty}^\infty \, d\mu_\eps(x,t)
\leq \norm{\eta(\ue(0,\cdot)}_{L^1(\R)}
\lesssim 1,
$$
assuming that $\norm{\ue}_{L^\infty(\R_+\times \R)}, 
\norm{\ue}_{L^\infty(\R_+;L^p(\R))}\lesssim 1$, $p=1,2$. 
To prevent ambiguity and ensure clarity, we denote the 
entropy/entropy flux pair as $(\eta, q)$, distinguishing 
them from the functions $(S,Q)$ used for \eqref{cauchy}. 
We denote by $\Delta_h W(x,t):=W(t,x+h)-W(x,t)$ the spatial 
difference operator with step size $h$. Set
$$
a :=\Delta_h \ue, \quad 
b :=\Delta_h f(\ue), \quad
C_a :=\eps \pxx \Delta_h\ue,
$$
and 
\begin{align*}
	& d:=\Delta_h \eta(\ue), \quad 
	e:=\Delta_h q(\ue),
	\quad 
	C_d:=C_{d,1}+C_{d,2},  
	\\ & \quad
	\text{where} \quad
	C_{d,1}:=\eps\pxx \Delta_h\eta(\ue),
	\quad 
	C_{d,2}:=-\Delta_h\mu_\eps.
\end{align*}
Then the system \eqref{eq:par-system} takes 
the form
\begin{align*}
	& a_t + b_x = C_a, 
	\\ & d_t+e_x=C_d,
\end{align*}
where, applying Lemma \ref{lemma_entropy_ineq} 
and Remark \ref{remark_eta_eq_f} 
with $S=\eta$ and $Q=q$, 
$$
ae-db \gtrsim \abs{\Delta_h \ue}^4=\abs{\ue(x+h,t)-\ue(x,t)}^4.
$$

We need the (spatial) anti-derivatives of $a$ and $d$:
$$
A(t,y):=\int_{-\infty}^y a(t,y)\, dx,
\quad 
D(x,t):=\int_x^{\infty} d(t,y)\, dy,
$$
which crucially are uniformly (in $\eps$) Lipschitz continuous. 

Denote by $I(t)$ the (spatial) interaction functional:
$$
I(t)=\iint\limits_{x<y} a(x,t)d(t,y)\,dx \, dy.
$$ 
A straightforward calculation will confirm that 
the following (spatial) interaction identify holds:
\begin{align*}
	\frac{d}{dt} I(t)
	& =\int_{\R} \bigl(ae-db \bigr)(t,z)\,dz
	\\ & \qquad 
	+\iint\limits_{x<y} 
	C_a(x,t)d(t,y)\,dx\,dy
	+\iint\limits_{x<y} 
	a(x,t)C_d(t,y)\,dx\,dy
	\\ & 
	=\int_{\R} \bigl(ae-db \bigr)(t,z)\,dz
	\\ & \qquad 
	+\int_{\R} C_a(x,t)D(x,t)\,dx
	+\int_{\R} A(t,y)C_d(t,y)\,dy.
\end{align*}

Since we assumed $a,d\in L^\infty_tL^1_x$,
$$
I(t)\le \norm{a}_{L^\infty_tL^1_x}
\norm{d}_{L^\infty_tL^1_x}\lesssim 1, 
\quad t\in \R_+.
$$
Hence, by integrating the above identify 
in $t\in [0,T]$, $T>0$, we arrive at
$$
\iint \abs{\ue(z+h,t)-\ue(z,t)}^4\, dz\,dt
\lesssim 1+ J_1+J_2+J_3,
$$
where
\begin{align*}
	& J_1 = -\iint \eps \pxx \Delta_h\ue 
	\left(\int_x^\infty\Delta_h \eta(\ue)\,dy\right)
	\,dx\,dt,
	\\ &
	J_2= -\iint \eps \pyy \Delta_h\eta(\ue) 
	\left(\int_{-\infty}^y \Delta_h \ue\,dx\right)
	\,dy\,dt,
	\\ &
	J_3= \iint (\Delta_h \mu_\eps)(t,y)
	\left(\int_{-\infty}^y \Delta_h 
	\ue\,dx\right)\,dy\,dt.
\end{align*}
Utilizing the assumed bounds on $\ue$, we can 
estimate these three terms following the approach 
outlined in \cite{Karlsen:2023aa}. The final result is 
that $\abs{J_i} \lesssim h$ for $i=1,2,3$, which implies that
$$
\iint  \abs{\ue(z+h,t+\tau)-\ue(z,t)}^4\, dz \dt \lesssim h.
$$
A similar temporal translation estimate can be derived for $\ue$.  
This brings us to the end of the outline 
detailing the proof for the vanishing viscosity 
approximation $\ue$.


\section{Preliminary results}\label{sec:prelim-est}

The following result is taken from \cite[Lemma 4.1]{Karlsen:2004sn}.
\begin{lemma}[monotonicity and $L^\infty$ estimate]
\label{th:monotonicity} 
Suppose the CFL condition \eqref{ass:CFLkappa} holds. 
Then the Lax-Friedrichs scheme \eqref{LxFr_scheme} is monotone. 
Moreover, the computed approximations satisfy 
$\ud(x,t) \in [a,b]$ for all $x$ and all $t\geq 0$.
\end{lemma}

In \cite{Karlsen:2004sn} we employed the single entropy $S(u) = u^2/2$. 
We wish to allow for an entropy of the form $S(k,u)$. 
The first part (inequality \eqref{eq:dissipation_new}) of the lemma that follows 
is a generalization of \cite[ Lemma 4.3]{Karlsen:2004sn}, which 
we prove by modifying the proof in \cite{Karlsen:2004sn}.
The second part of the lemma (inequality \eqref{ent_abs}) 
is new, and is required for the analysis that follows.

\begin{lemma}\label{lem:dissipation_new}
Let $(S,Q)$ be an arbitrary entropy/entropy 
flux pair defined by \eqref{def_ent_pair}, satisfying 
the uniform convexity condition \eqref{S_convex}. 
With $k_{j-1}^n,U_{j-1}^n$ and $k_{j+1}^n,U_{j+1}^n$ 
given, compute $U_j^{n+1}$ by \eqref{LxFr_scheme}. Define
\begin{equation*}
	\begin{split}
		\Psi_j^n
		&= S\left(k_j^{n+1},U_j^{n+1}\right) 
		- \frac12 \left(S\left(k_{j-1}^{n},U_{j-1}^{n}\right)
		+S\left(k_{j+1}^{n},U_{j+1}^{n}\right)\right)\\
		& \qquad + \frac{\lambda}{2} \left(Q\left(k_{j+1}^n,U_{j+1}^n\right)
		-Q\left(k_{j-1}^n ,U_{j-1}^n\right)\right), \quad j \in \Omega_n'.
	\end{split}
\end{equation*}
Then
\begin{equation}
   \label{eq:dissipation_new}
   \Psi_j^n
      \le - \frac{\gamma \kappa^2}{8} \left(U_{j+1}^n - U_{j-1}^n\right)^2
      + \lambda K_1 \abs{k_{j+1}^n - k_{j-1}^n}
      +\tilde{K}_1 \abs{k_{j}^{n+1} - k_{j-1}^n},
\end{equation}
where $K_1$ is a constant that is independent of $\Dx$.

In addition, for some constants $K_2$, $K_3$, and $K_4$ 
that are independent of $\Dx $,
\begin{equation}\label{ent_abs}
\abs{\Psi_j^n }
      \le K_2 \abs{U_{j+1}^n - U_{j-1}^n}^2 + K_3 \abs{k_{j+1}^n - k_{j-1}^n} 
      + K_4 \abs{k_j^{n+1}- \f12 \left(k_{j-1}^n + k_{j+1}^n \right)}.
\end{equation}
The estimate \eqref{ent_abs} holds without the 
convexity condition \eqref{S_convex}. 
\end{lemma}

\begin{proof}
Let us introduce the functions $w,v,\Phi:[a,b]\to \R$ defined by
\begin{equation*}
	\begin{split}
		w(s) &= s U_{j-1}^n +(1-s)U_{j+1}^n,\\
      		v(s) &= \frac12\left(w(s)+U_{j+1}^n\right)
      		-\frac{\lambda}{2}\left(f\left(k_{j+1}^n,U_{j+1}^n\right)
      		-f\left(k_{j-1}^n,w(s)\right)\right),\\
      		\Phi(s) &= \frac12\left( S(k_{j-1}^n,w(s))
		+ S\left(k_{j+1}^n,U_{j+1}^n\right)\right)
		\\ & \qquad\qquad 
		+\frac{\lambda}{2}\left(Q\left(k_{j-1}^n,w(s)\right)
		-Q\left(k_{j+1}^n,U_{j+1}^n\right)\right) -S(k_{j-1}^n,v(s)).
	\end{split}
\end{equation*}
We collect the following elementary facts about
these functions 
in one place before continuing with the proof:
\begin{equation*}
	\begin{split}
		w(0) &= U_{j+1}^n, \quad w(1) = U_{j-1}^n, \quad w'(s)
		=U_{j-1}^n - U_{j+1}^n, \\
		v(0) &= U_{j+1}^n-\frac{\lambda}{2}
		\left(f\left(k_{j+1}^n,U_{j+1}^n\right)
		-f\left(k_{j-1}^n,U_{j+1}^n\right)\right),
		\quad v(1) = U_j^{n+1}, \\
		w'(s) & = \frac12\left(1+\lambda \pu f\left(k_{j-1}^n,w(s)\right)\right)
		\left(U_{j-1}^n - U_{j+1}^n\right),\\
		v'(s) &=\frac12\left(1+\lambda \pu f\left(k_{j-1}^n,w(s)\right)\right),\\
		\Phi(0) &= \f12 S\left(k_{j-1}^n,U_{j+1}^n\right)
		+\f12 S\left(k_{j+1}^n,U_{j+1}^n\right)
		-\frac{\lambda}{2} \left(Q\left(k_{j+1}^n,U_{j+1}^n\right)
		-Q\left(k_{j-1}^n,U_{j+1}^n\right)\right)
		\\ & \qquad \qquad \qquad
		-S\left(k_{j-1}^n,U_{j+1}^n
		-\frac{\lambda}{2}\left(f\left(k_{j+1}^n,U_{j+1}^n\right)
		-f\left(k_{j-1}^n,U_{j+1}^n\right)\right)\right),\\
		\Phi(1) &= \frac12\left(S\left(k_{j-1}^n,U_{j-1}^n\right)
		+S\left(k_{j+1}^n,U_{j+1}^n\right) \right)
		\\ & \qquad\qquad
		+\frac{\lambda}{2}\left(Q\left(k_{j-1}^n,U_{j-1}^n\right)
		-Q\left(k_{j+1}^n,U_{j+1}^n\right)\right)
		-S\left(k_{j-1}^n,U_j^{n+1}\right).
	\end{split}
\end{equation*}

We seek an estimate of $\Phi'(s)$.
It is readily verified that
\begin{equation*}
	\Phi'(s) =\frac12 
	\left(1 + \lambda \pu f\left(k_{j-1}^n,w(s)\right)\right) \mG
	\left(U_{j-1}^n - U_{j+1}^n\right) (w(s)-v(s)),
\end{equation*}
where
\begin{equation*}
	\mG := 
	\begin{cases}
		{\pu S(k_{j-1}^n, w(s)) - \pu S(k_{j-1}^n,v(s)) \over w(s) - v(s)}, 
		&\quad w(s) \ne v(s),\\
		\puu S(k_{j-1}^n, w(s)), &\quad w(s) = v(s),
	\end{cases}
\end{equation*}
and
\begin{align*}
	w(s)-v(s)
	& = \frac12\left(1- \lambda\frac{f\left(k_{j-1}^n,w(s)\right)
	-f\left(k_{j-1}^n,U_{j+1}^n\right)}{w(s)-U_{j+1}^n}\right)
	\left(w(s)-U_{j+1}^n\right) 
	\\ & \qquad \qquad
	+\frac{\lambda}{2} \left(f\left(k_{j+1}^n,U_{j+1}\right)
	-f\left(k_{j-1}^n,U_{j+1}\right)\right),
\end{align*}
so that, with
$\mF:=1- \lambda \left(f\left(k_{j-1}^n,w(s)\right)
-f\left(k_{j-1}^n,U_{j+1}^n\right)\right)/\left(w(s)-U_{j+1}^n\right)$,
\begin{equation*}
	\begin{split}
		w(s)-v(s) &= \frac{\mF}{2} \left(w(s)-U_{j+1}^n\right)
		+\frac{\lambda}{2} \left(f\left(k_{j+1}^n,U_{j+1}^n\right)
		-f\left(k_{j-1}^n,U_{j+1}^n\right) \right)
		\\ &=\frac{\mF}{2}\left(U_{j-1}^n-U_{j+1}^n\right)s
		+\frac{\lambda}{2}\left(f\left(k_{j+1}^n,U_{j+1}^n\right)
		-f\left(k_{j-1}^n,U_{j+1}^n\right) \right).
	\end{split}
\end{equation*}
This yields 
\begin{equation}\label{Phi_prime}
	\begin{split}
		\Phi'(s)&={1 \over 4} \mF \mG \left(1 + \lambda \pu f(k_{j-1}^n,w(s)) \right) 
		\left(U_{j-1}^n-U_{j+1}^n \right)^2 s 
		\\ & \qquad
		+ {\lambda \over 4}  \mG \left(1 + \lambda \pu f(k_{j-1}^n,w(s)) \right) 
		\left(U_{j-1}^n-U_{j+1}^n \right)
		\\ & \qquad \qquad 
		\times \left(f(k_{j+1}^n,U_{j+1}^n) - f(k_{j-1}^n,U_{j+1}^n) \right).
	\end{split}
\end{equation}

\noindent
As a consequence of the CFL condition \eqref{ass:CFLkappa}, 
$\mF \geq \kappa$, and
\begin{equation*}
	1 + \lambda \pu f\left(k_{j-1}^n,w(s)\right) 
	\ge \kappa, \quad \abs{1 + \lambda \pu f\left(k_{j-1}^n,w(s)\right)} \le 2.
\end{equation*}
Also, $\mG \ge \gamma$ and $\abs{\mG} \le \norm{\puu S}$. Thus,
\begin{equation*}
	\Phi'(s) \geq \frac{\gamma \kappa^2}{4}
	\left(U_{j-1}^n-U_{j+1}^n\right)^2 s -
	{\lambda \over 2} (b-a) 
	\norm{\puu S} \norm{\pk f}\abs{k_{j+1}^n - k_{j-1}^n}.
\end{equation*}

\noindent
Integrating this last inequality from $0$ to $1$ gives
\begin{equation*}
	\Phi(1) - \Phi(0) \geq
	\frac{\gamma \kappa^2}{8} \left(U_{j-1}^n-U_{j+1}^n\right)^2
	- {\lambda \over 2} (b-a) \norm{\puu S} 
	\norm{\pk f}\abs{k_{j+1}^n - k_{j-1}^n}.
\end{equation*}
Combining this with the fact that
\begin{equation*}
	\Phi(1) = - \Psi_j^n + S(k_j^{n+1},U_j^{n+1}) 
	- S(k_{j-1}^n,U_j^{n+1}),
\end{equation*}
we obtain
\begin{equation*}
	\begin{split}
		\Psi_j^n 
		&\le -\frac{\gamma \kappa^2}{8} \left(U_{j-1}^n-U_{j+1}^n\right)^2
		+ {\lambda \over 2} (b-a) \norm{\puu S} 
		\norm{\pk f}\abs{k_{j+1}^n - k_{j-1}^n}
		\\ & \qquad 
		+S(k_j^{n+1},U_j^{n+1}) - S(k_{j-1}^n,U_j^{n+1}) - \Phi(0).
	\end{split}
\end{equation*}
Using $\abs{S(k_j^{n+1},U_j^{n+1}) - S(k_{j-1}^n,U_j^{n+1})} 
\le \norm{\partial_k S}\abs{k_{j}^{n+1} - k_{j-1}^{n}}$, the 
proof of  \eqref{eq:dissipation_new} will be complete as soon
as we show that $\abs{\Phi(0)}$ is bounded by a constant
times $\abs{k_{j+1}^n - k_{j-1}^n}$.
A straightforward calculation results in
\begin{equation*}
	\begin{split}
		\Phi(0)  &= S(k_{j-1}^n, U_{j+1}^n)
		-S\left(k_{j-1}^n, U_{j+1}^n
		-{\lambda \over 2} (f(k_{j+1}^n,U_{j+1}^n) - f(k_{j-1}^n,U_{j+1}^n))\right)
		\\ & \qquad 
		+\f12 \left(S(k_{j+1}^n,U_{j+1}^n) - S(k_{j-1}^n,U_{j+1}^n) \right)
		-{\lambda \over 2} 
		\left(Q(k_{j+1}^n,U_{j+1}^n)-Q(k_{j-1}^n,U_{j+1}^n) \right),
	\end{split}
\end{equation*}
which yields the inequality
\begin{equation}\label{Phi_0_est}
	\abs{\Phi(0)} \le \left( {\lambda \over 2} \norm{\pu S} \norm{\pk f}
	+ \f12 \norm{\pk S}
	+{\lambda \over 2}  \norm{\pk Q} \right) 
	\abs{k_{j+1}^n - k_{j-1}^n}.
\end{equation}
Recalling that
$\abs{k_{j+1}^n - k_{j-1}^n} \le \beta - \alpha$, cf.~\eqref{ass:k},
the inequality \eqref{Phi_0_est} provides the desired bound 
on $\abs{\Phi(0)}$ and the proof of \eqref{eq:dissipation_new} 
is complete.

For the proof of \eqref{ent_abs},
\begin{equation*}
	\begin{split}
		&\Biggl|S\left(k_j^{n+1},U_j^{n+1}\right) 
		- \frac12 \left(S\left(k_{j-1}^{n},U_{j-1}^{n}\right)
		+S\left(k_{j+1}^{n},U_{j+1}^{n}\right)\right)
		\\ & \qquad\qquad 
		+ \frac{\lambda}{2} \left(Q\left(k_{j+1}^n,U_{j+1}^n\right)
		-Q\left(k_{j-1}^n ,U_{j-1}^n\right)\right)\Biggr| 
		\\ & \quad 
		= \abs{\Phi(1) + S(k_{j}^{n+1},U_j^{n+1}) - S(k_{j-1}^n,U_j^{n+1})}
		\le  \abs{\Phi(1)}+\norm{\pk S} \abs{k_{j}^{n+1} - k_{j-1}^n}
		\\ &\quad \le  \abs{\Phi(1)}+\norm{\pk S} \left( 
		\f12 \abs{k_{j+1}^n - k_{j-1}^n} + \abs{k_j^{n+1}- \f12 
		\left(k_{j+1}^n + k_{j-1}^n \right)}\right).
	\end{split}
\end{equation*}
The proof of \eqref{ent_abs} will be complete if we can 
obtain a suitable estimate of $\abs{\Phi(1)}$.
We have
\begin{equation*}
	\abs{\Phi(1)} \le \max_{s \in [0,1]} \abs{\Phi'(s)} + \abs{\Phi(0)}.
\end{equation*}
Recalling \eqref{Phi_0_est}, it suffices to estimate 
$\max_{s \in [0,1]} \abs{\Phi'(s)}$.
Referring back to \eqref{Phi_prime}, and using 
$\abs{\mF} \le 2$, $\abs{\mG} \le \norm{\puu S}<\infty$, 
cf.~\eqref{ass:S_Q_1}, we find that
\begin{equation*}
	\max_{s \in [0,1]} \abs{\Phi'(s)}
	\le \norm{\puu S} \abs{U_{j+1}^n - U_{j-1}^n}^2
	+{\lambda \over 2} \norm{\puu S} (b-a) 
	\norm{\pk f} \abs{k_{j+1}^n - k_{j-1}^n},
\end{equation*}
which completes the proof of \eqref{ent_abs}.
\end{proof}

In what follows, we consider an arbitrary weight function 
$\chi \in C^1(\R) \cap L^1(\R)$ satisfying 
the properties in \eqref{weight_properties}. 
We will use the following abbreviations:
 \begin{equation*}
 	f_j^n = f(k_j^n,U_j^n), \quad 
	S_j^n = S(k_j,U_j^n), \quad 
	Q_j^n = Q(k_j^n,U_j^n), \quad 
	\chi_j = \chi(x_j).
 \end{equation*}
 Also, the following facts will be helpful.
For a fixed integer $i$, if $\{Z_j^n \}$ is bounded, then
\begin{equation*}
\Dx  \sumj \chi_{j+i} Z^n_j < \infty,
\end{equation*}
and
\begin{equation*}
\Dx  \sumj \chi_{j+i} Z^n_j 
= \Dx  \sumj \chi_j Z^n_j +O(i \Dx ).
\end{equation*}

The following result is the version of  \cite[Lemma 4.3]{Karlsen:2004sn}
that is appropriate for the assumptions of this paper.

\begin{lemma}\label{lemma_trans_sqd}
For $T>0$, $N = \lfloor T/ \Dt \rfloor$, we have the bounds
\begin{equation}
   \label{eq:tvsquared_weight}
   \begin{split}
      &\Dx \sum_{n=0}^N
      \sumjp
      \chi_{j-1} \left(U_{j+1}^n - U_{j-1}^n\right)^2\le C_1(T) < \infty,
      \\ & \sum_{n =0}^{N} \int_{\R}
      \chi(x) \left(u^{\D}(x,t^{n+1})-u^{\D}(x,t^n)  \right)^2 \dx
      \le C_2(T) < \infty,
   \end{split}
\end{equation}
where $C_1(T)$ and $C_2(T)$ are independent of $\Delta$.
\end{lemma}

\begin{proof}
For the first inequality of \eqref{eq:tvsquared_weight}, we use
\eqref{eq:dissipation_new} with $S=u^2/2$:
\begin{equation}\label{tvs_wt_1}
\begin{split}
&{\gamma \kappa^2 \over 8} \Dx \sum_{n=0}^N
\sumjp \chi_{j-1} \left(U_{j+1}^n - U_{j-1}^n\right)^2 
\\ & \quad \le \Dx \sum_{n=0}^N\sumjp  \chi_{j-1}
\left(- S_j^{n+1} + \f12 (S_{j-1}^n + S_{j+1}^n) \right) 
\\ & \quad \qquad 
- {\lambda \over 2} \Dx \sum_{n=0}^N
\sumjp \chi_{j-1}  (Q_{j+1}^n - Q_{j-1}^n)
 + \lambda K_1 \Dx \sum_{n=0}^N
 \sumjp \chi_{j-1} \abs{k_{j+1}^n - k_{j-1}^n}
 \\ & \quad \qquad
+ \tilde{K}_1 \Dx \sum_{n=0}^N 
\sumjp \chi_{j-1}  \abs{k_{j}^{n+1} - k_{j-1}^n}.
\end{split}
\end{equation}

For the first sum on the right side of \eqref{tvs_wt_1} 
we use the identity
\begin{equation*}
\begin{split}
&\chi_{j-1} \left(- S_j^{n+1} + \f12 (S_{j-1}^n + S_{j+1}^n) \right)
\\ & \quad 
= - \chi_j S_j^{n+1} + \f12 (\chi_{j-1} S_{j-1}^n + \chi_{j+1}S_{j+1}^n) 
+ S_j^{n+1} (\chi_j - \chi_{j-1}) - \f12 S_{j+1}^n(\chi_{j+1} - \chi_{j-1}).
\end{split}
\end{equation*}
Summing over $n$ and $j$, the
terms $- \chi_j S_j^{n+1} + \f12 (\chi_{j-1} S_{j-1}^n 
+ \chi_{j+1}S_{j+1}^n)$ telescope. The result is
\begin{equation}\label{tvs_wt_3}
\begin{split}
& \Dx \sum_{n=0}^N\sumjp  \chi_{j-1}
\left(- S_j^{n+1} + \f12 (S_{j-1}^n + S_{j+1}^n) \right) 
\\ & \quad 
= -\Dx \sum_{j \in \Omega_{N+1}} \chi_j S_j^{N+1}
+\Dx  \sum_{j \in \Omega_0} \chi_j S_j^0\\
& \quad \qquad
+\Dx \sum_{n=0}^N \sumjp S_j^{n+1}(\chi_j - \chi_{j-1})
-\f12 \Dx \sum_{n=0}^N
\sumjp S_{j+1}^{n}(\chi_{j+1} - \chi_{j-1}).
\end{split}
\end{equation}
The desired bound for the first sum on the right side of \eqref{tvs_wt_1} 
then follows from \eqref{tvs_wt_3}, using \eqref{weight_properties}.

For the second sum on the right side of \eqref{tvs_wt_1} 
we use the identity
\begin{equation}\label{tvs_wt_4}
\chi_{j-1}(Q_{j+1}^n - Q_{j-1}^n)
 =(\chi_{j+1} Q_{j+1}^n - \chi_{j-1} Q_{j-1}^n) 
 - Q_{j+1}^n (\chi_{j+1} - \chi_{j-1}).
\end{equation}
Summing over $n$ and $j$, the contribution 
from the first part of the right side of \eqref{tvs_wt_4}
telescopes, resulting in
\begin{equation}\label{tvs_wt_5}
\Dx \sum_{n=0}^N
\sumjp  \chi_{j-1}  (Q_{j+1}^n - Q_{j-1}^n)
= -\Dx \sum_{n=0}^N
\sumjp  Q_{j+1}^n (\chi_{j+1} - \chi_{j-1}).
\end{equation}
The desired bound for the second sum then 
follows from \eqref{tvs_wt_5}, using \eqref{weight_properties}.

The third and fourth sums on the right side of \eqref{tvs_wt_1} 
are bounded in absolute value due to the fact 
that $\chi$ is bounded and $k \in BV(\R\times\R_+)$, and 
the proof of the first part of \eqref{eq:tvsquared_weight} is complete.

For the second part of \eqref{eq:tvsquared_weight},  
\begin{equation}\label{tvs_wt_6}
	\begin{split}
		&\sumnN \int_{\R}
		\chi(x) \left(u^{\D}(x,t^{n+1})-u^{\D}(x,t^n)  \right)^2 \dx
		\\ & \quad 
		= \sumnN \sumjp \int_{x_{j-1}}^{x_j} \chi(x) 
		\left(U_j^{n+1}-U_{j-1}^n \right)^2 \dx
		\\ & \quad \qquad 
		+\sumnN \sumjp \int_{x_j}^{x_{j+1}} \chi(x) 
		\left(U_j^{n+1}-U_{j+1}^n \right)^2 \dx.
	\end{split}
\end{equation}
We prove a bound for the first sum on the right side 
of \eqref{tvs_wt_6}. A bound for the second
sum is proven in a similar manner. To this end,
\begin{equation}\label{tvs_wt_7}
\begin{split}
 \left(U_j^{n+1}-U_{j-1}^n \right)^2
 &= \left(U_j^{n+1}- \f12 (U_{j-1}^n + U_{j+1}^n) 
 + \f12 (U_{j+1}^n - U^n_{j-1}) \right)^2\\
 & \le 2  \left(U_j^{n+1}- \f12 (U_{j-1}^n 
 + U_{j+1}^n)\right)^2 
 + \f12 \left(U_{j+1}^n - U^n_{j-1} \right)^2.
\end{split}
\end{equation}
Here we have used the inequality $(a+b)^2 \le 2 (a^2 + b^2)$.
Recalling \eqref{LxFr_scheme}, we can replace \eqref{tvs_wt_7} by
\begin{equation*}
\begin{split}
& \left(U_j^{n+1}-U_{j-1}^n \right)^2 \\
 &\qquad \le 2 \lambda^2 
 \left(f(k_{j+1}^n,U_{j+1}^n) - f(k_{j-1}^n,U_{j-1}^n) \right)^2 
 + \f12 \left(U_{j+1}^n - U_{j-1}^n \right)^2.
\end{split}
\end{equation*}
We use 
\begin{equation*}
	\abs{f(k_{j+1}^n,U_{j+1}^n) - f(k_{j-1}^n,U_{j-1}^n)}
	\le \norm{\pu f}\abs{U_{j+1}^n - U_{j-1}^n} 
	+ \norm{\pk f}\abs{k_{j+1}^n - k_{j-1}^n},
\end{equation*}
along with an application of Young's inequality, 
$ab \le (a^2 + b^2)/2$, to obtain
\begin{equation}\label{tvs_wt_9}
	\begin{split}
		 \left(U_j^{n+1}-U_{j-1}^n \right)^2 
		 &\le d_1 \left(U_{j+1}^n - U_{j-1}^n \right)^2 
		 + d_2 \left(k_{j+1}^n - k^n_{j-1} \right)^2\\
		 &\le d_1 \left(U_{j+1}^n - U_{j-1}^n \right)^2 
		 + d_2 (\beta - \alpha)\abs{k_{j+1}^n - k^n_{j-1}},
	\end{split}
\end{equation}
where $d_1$ and $d_2$ are constants that 
are independent of $\Dx $.

The proof of the second part of \eqref{eq:tvsquared_weight} is 
completed by substituting \eqref{tvs_wt_9} into \eqref{tvs_wt_6}, 
along with a similar inequality for the second sum on
the right side of \eqref{tvs_wt_6}. The desired inequality then 
follows from the first part of \eqref{eq:tvsquared_weight}, the 
fact that $k \in BV(\R\times\R_+)$, and
\begin{equation*}
	\sumjp \int_{x_{j-1}}^{x_{j+1}} \abs{\chi_{j-1} - \chi(x)} \dx 
	\lesssim \Dx  \int_\R \chi(x) \, dx,
\end{equation*}
which follows from \eqref{weight_properties}.
\end{proof}

The following lemma is a mild adaptation 
of \cite[Lemma 5.2]{Golse:2013aa}. 
It will be used in the upcoming two sections.

\begin{lemma}\label{lemma_entropy_ineq}
Suppose the flux $f$ satisfies \eqref{ass:f_en}, \eqref{ass:f_to}, 
and the genuine nonlinearity assumption in \eqref{f_S_assumptions}. 
Suppose also that we are given an entropy/entropy flux 
pair $(S,Q)$ that satisfies \eqref{def_ent_pair}, \eqref{ass:S_Q_1}, 
and \eqref{f_S_assumptions}. Then for all 
$v,w \in [a,b]$ and $k \in [\alpha,\beta]$,
\begin{equation}\label{f_S_ineq}
	\begin{split}
		&(w-v)\left(Q(k,w)-Q(k,v) \right) 
		\\ & \qquad 
		- \left(S(k,w)-S(k,v) \right) \left(f(k,w)-f(k,v) \right)
		\ge C_{f,S} \abs{w-v}^{p_f + p_{S} + 2},
	\end{split}
\end{equation}
where $C_{f,S} = {C_f C_S 
\over (1+p_f+p_S) (2 + p_f + p_S)}$.
\end{lemma}

\begin{remark}\normalfont 
\label{remark_eta_eq_f}
Suppose $f \in C^2([\alpha,\beta] \times [a,b])$ 
and $\puu f(k,\xi) \ge \gamma >0$ for all 
$(k,\xi) \in [\alpha,\beta] \times [a,b]$. If we 
choose $S = f$ as the entropy, then 
$C_f=C_S=\gamma$, $p_f = p_S = 1$, 
$C_{f,S}=\gamma^2/12$, and thus
\begin{equation*}
	(w-v)\left(Q(k,w)-Q(k,v) \right)
	-\left(f(k,w)-f(k,v) \right)^2
	\ge {\gamma^2 \over {12}} \abs{w-v}^4.
\end{equation*}
\end{remark}

\begin{proof}
A straightforward calculation using \eqref{def_ent_pair} yields
\begin{equation}\label{cvx_ineq_1}
\begin{split}
&\int_v^w \int_v^w \pu S(k,\zeta)\bigl(\pu f(k,\zeta)-\pu f(k,\xi) \bigr) \, d\xi \,d\zeta\\
&\qquad = (w-v) \left(Q(k,w)-Q(k,v) \right)
-\left(S(k,w)-S(k,v) \right)  \left(f(k,w)-f(k,v) \right).
\end{split}
\end{equation}
Let 
\begin{equation*}
\Lambda(\xi,\zeta) = \bigl(\pu S(k,\zeta)-\pu S(k,\xi)\bigr)
\bigl(\pu f(k,\zeta)-\pu f(k,\xi) \bigr).
\end{equation*}
For now assume that $w\ge v$. Let $R:=[v,w]\times [v,w]$, 
and write $R = R_1 \cup R_2$, where
\begin{equation*}
R_1=\{ (\xi,\zeta) \in R: \xi>\zeta \}, \quad 
R_2 =\{ (\xi,\zeta) \in R: \xi \le \zeta \}.
\end{equation*}
By combining \eqref{cvx_ineq_1} with the version that 
results by swapping $\xi$ and $\zeta$, we obtain
\begin{equation}\label{cvx_ineq_2}
\begin{split}
&(w-v) \left(Q(k,w)-Q(k,v) \right) - \left(S(k,w)-S(k,v) \right)
\left(f(k,w)-f(k,v) \right)\\
&\qquad =\f12 \dint_R
\Lambda(\xi,\zeta)\, d\xi \,d\zeta.
\end{split}
\end{equation}
Note that $\Lambda$ is symmetric, i.e., 
$\Lambda(\xi,\zeta)= \Lambda(\zeta,\xi)$. Thus
\begin{equation}\label{cvx_ineq_2a}
\dint_R \Lambda(\xi,\zeta)  \, d\xi \,d\zeta
=2\dint_{R_2} \Lambda(\xi,\zeta)\, d\xi \,d\zeta.
\end{equation}

Recalling \eqref{f_S_assumptions}, for $(\xi,\zeta) \in R_2$ 
we have the inequality
\begin{equation}\label{cvx_ineq_3}
\Lambda(\xi,\zeta)
\ge C_S C_f (\zeta - \xi)^{p_S + p_f}.
\end{equation}
By combining \eqref{cvx_ineq_2}, \eqref{cvx_ineq_2a} 
and \eqref{cvx_ineq_3}, we obtain
\begin{equation*}
\begin{split}
&(w-v) \left(Q(k,w)-Q(k,v) \right)
-\left(S(k,w)-S(k,v) \right)  \left(f(k,w)-f(k,v) \right)
\\ & \qquad 
\ge C_S C_f \dint_{R_2}  
(\zeta - \xi)^{p_S + p_f} \, d\xi \,d\zeta
= C_{f,S} (w-v)^{p_f + p_{S} + 2}.
\end{split}
\end{equation*}
This completes the proof for the case where $v\le w$. 
We obtain the desired inequality for $w \le v$ 
by swapping the roles of $v$ and $w$.
Finally, combining the two cases yields \eqref{f_S_ineq}.
\end{proof}

\section{Estimate of spatial  translates}
\label{sec:spatial_est}

This section focuses on stating and proving 
Proposition~\ref{prop_space_est}, crucial for deriving 
Theorem~\ref{thm:main} as mentioned earlier. To facilitate this, we 
will first lay the groundwork with a series of technical lemmas. 
A condensed version of the proof, particularly addressing 
the viscosity approximation, is presented at the end 
of Section \ref{sec:LxFr} to guide the upcoming calculations.

For now we assume that
$h = 2 \nu$, where $\nu$ is a nonnegative integer.
In what follows $\D_h$ and $\D_{-h}$ denote 
the spatial finite difference operators,
\begin{equation*}
\D_h Z_j^n = Z_{j+2\nu}^n - Z_j^n, 
\quad \D_{-h}  Z_j^n = Z_{j}^n - Z_{j-2\nu}^n.
\end{equation*}
We will also use the difference operators $\D^u_h$ and $\D^h_h$ defined by
\begin{equation*}
\begin{split}
&\D^u_h f_j^n = \D^u_h f(k_j^n,U_j^n) 
= f(k_j^n,U_{j+2\nu}^n)-f(k_j^n,U_j^n), \\
&\D^k_h f_j^n = \D^k_h f(k_j^n,U_j^n) 
= f(k_{j+2\nu}^n,U_{j}^n)-f(k_j^n,U_j^n),
\end{split}
\end{equation*}
with similar definitions for $\D^u_{-h} f_j^n$ and $\D^k_{-h} f_j^n$.
Note the identity
\begin{equation}\label{Dh_identity}
\D_h f_j^n = \D^u_h f_j^n + \D^k_{-h} f_{j+2\nu}^n.
\end{equation}
We will employ the inequalities 
\begin{equation}\label{D_inequalities}
\begin{split}
&\abs{\D_h U_j^n} \le b-a,
\\ &\abs{\D_h G_j^n} \le \norm{\partial_u G} 
 \abs{U_{j+2\nu}^n - U_j^n}
+\norm{\partial_k G}  \abs{k_{j+2\nu}^n - k_j^n}, \quad
\textrm{for $G = f, S, Q$}.
\end{split}
\end{equation}

The immediate goal is to prove Lemma~\ref{lemma_space_trans} 
below. We will use the following key fact, which 
is the content of Lemma~\ref{lemma_entropy_ineq}:
\begin{equation}\label{Gamma_estimate}
\Gamma_j^n \ge C_{f,S} 
\abs{U_{j+2 \nu}^n - U_{j}^n}^{p_f+p_{S}+2},
\end{equation}
where
\begin{equation}\label{def_Gamma}
\begin{split}
\Gamma_j^n &=\left(U_{j+2\nu}^n-U_{j}^n\right) 
\left(Q(k_{j}^n,U_{j+2\nu}^n) - Q(k_{j}^n,U_{j}^n)\right)\\
& \qquad -\left(S(k_j^n, U_{j+2\nu}^n)-S(k_j^n, U_{j}^n)\right)
\left(f(k_{j}^n,U_{j+2\nu}^n) - f(k_{j}^n,U_{j}^n) \right).
\end{split}
\end{equation}

We write the Lax-Friedrichs scheme and entropy inequality in the form
\begin{equation}\label{lxf_sys_1}
\left\{
\begin{split}
  &U_j^{n+1}-\frac12\left(U_{j-1}^n+U_{j+1}^n\right)
  + \frac{\lambda}{2} \left(f_{j+1}^n  - f_{j-1}^n \right) = 0, 
  \quad j \in \Omega'_n,\\
   &S_j^{n+1}-\frac12\left(S_{j-1}^n+S_{j+1}^n\right)
   + \frac{\lambda}{2} \left(Q_{j+1}^n  - Q_{j-1}^n \right) =  \Psi_j^n,  
   \quad j \in \Omega'_n,
\end{split}
\right.
\end{equation}
where, by \eqref{ent_abs} of Lemma~\ref{lem:dissipation_new},
\begin{equation*}
\abs{\Psi_j^n} \le K_2 \abs{U_{j+1}^n - U_{j-1}^n}^2
+ K_3 \abs{k_{j+1}^n - k_{j-1}^n}
+K_4 \abs{k_j^{n+1}- \f12 
\left(k_{j-1}^n + k_{j+1}^n \right)}.
\end{equation*}
With the assumption that $k \in BV(\R\times\R_+)$, 
along with \eqref{eq:tvsquared_weight} of Lemma~\ref{lemma_trans_sqd},
\begin{equation}\label{Psi_est}
\Dx  \sumnN  \sumjp \chi_j \abs{\Psi_j^n} \le C_3(T),
\end{equation}
where $C_3(T)$ is a constant that is independent of $\Dx $.

We multiply both equations of \eqref{lxf_sys_1}
by $\chi_j \D_h$, and after some algebra, arrive at the 
following $2\times2$ discrete system:
\begin{equation}\label{lxf_sys_2}
\left\{
\begin{split}
 A_j^{n+1} - \f12 \left(A_{j-1}^n + A_{j+1}^n \right) 
 +{\lambda \over 2} \left(B_{j+1}^n - B_{j-1}^n \right) = C_{A,j}^n, 
 \quad j \in \Omega'_n,\\
 D_j^{n+1} - \f12 \left(D_{j-1}^n + D_{j+1}^n \right) 
 +{\lambda \over 2} \left(E_{j+1}^n - E_{j-1}^n \right) = C_{D,j}^n, 
 \quad j \in \Omega'_n,
\end{split}
\right.
\end{equation}
where
\begin{equation}\label{def_ABCD_1}
\begin{split}
&A_j^n = \chi_j \D_h U_j^n, \quad B_j^n = \chi_j \D_h f_j^n,
\\
&D_j^n = \chi_j \D_h S_j^n, \quad E_j^n = \chi_j \D_h Q_j^n,
\\ 
&C_{A,j}^n =  - \f12 \D_h U_{j+1}^n \left(\chi_{j+1} - \chi_j \right)
+ \f12 \D_h U_{j-1}^n \left(\chi_{j} - \chi_{j-1} \right)
\\ &\qquad \qquad 
+ {\lambda \over 2} \D_h f_{j+1}^n \left(\chi_{j+1} - \chi_j \right)
+ {\lambda \over 2} \D_h f_{j-1}^n \left(\chi_{j} - \chi_{j-1} \right),
\\
&C_{D,j}^n = - \f12 \D_h S_{j+1}^n \left(\chi_{j+1} - \chi_j \right)
+ \f12 \D_h S_{j-1}^n \left(\chi_{j} - \chi_{j-1} \right)\\
&\qquad \qquad + {\lambda \over 2} \D_h Q_{j+1}^n 
\left(\chi_{j+1} - \chi_j \right)
+ {\lambda \over 2} \D_h Q_{j-1}^n 
\left(\chi_{j} - \chi_{j-1} \right)
\\ &\qquad \qquad + \chi_j \D_h \Psi_j^n.                 
\end{split}
\end{equation}

Define
\begin{equation}\label{def_mA_mD}
\begin{split}
&\mA_\ell^n =  \Dx  \sum_{\Omega_n:j\le \ell} A_j^n,
 \quad \ell \in \Omega_n,
 \\ &
 \mD_j^n = \Dx  \sum_{\Omega_n:\ell\ge j} D_\ell^n, 
 \quad j \in \Omega_n,
\end{split}
\end{equation}
where 
\begin{equation*}
\begin{split}
&\text{``$\Omega_n:j\le \ell$'' 
means $\{j \in \Omega_n: j\le \ell\}$}, 
\quad
\text{``$\Omega_n:\ell \ge j$'' 
means $\{\ell \in  \Omega_n: \ell \ge j\}$}.
\end{split}
\end{equation*}
The following identities hold:
\begin{equation}\label{AD_identities}
\begin{split}
& \Dx^2 \dsum_{\Omega_n:\ell \ge j} C_{A,j}^n D_\ell^n 
= \Dx \sumj C_{A,j}^n \mD_j^n, 
\\ & 
\Dx^2 \dsum_{\Omega_n:\ell \ge j} C_{D,\ell}^n A_j^n 
= \Dx \suml C_{D,\ell}^n \mA_\ell^n,
\end{split}
\end{equation}
where we are using the notational convention
\begin{equation*}
\dsum_{\Omega_n:\ell \ge j} P_{j,\ell}
= \dsum_{\{(j,\ell) \in 
\Omega_n \times \Omega_n : \ell \ge j \}} P_{j,\ell},
\end{equation*}
with a similar definition for sums where 
$\Omega_n$ is replaced by $\Omega_n'$.

The interaction identity \eqref{lxf_interaction} below can be viewed as
a discrete, Lax-Friedrichs version of the  interaction 
identity (7) in \cite{Golse:2013aa}, and also
the stochastic interaction identity (3.12) 
of \cite{Karlsen:2023aa} (with $\sigma =0$). 
The term $\mE$ appearing on the right side 
of \eqref{lxf_interaction} is purely an artifact of discretization,
and does not have a counterpart in the interaction 
identities of \cite{Golse:2013aa} and \cite{Karlsen:2023aa}.

\begin{lemma}[Lax-Friedrichs interaction identity, spatial]
\label{lemma_lxf_interaction}
Define
\begin{equation}\label{def_In}
	I^n = \Dx^2 \dsum_{\Omega_n: j \le \ell} A_j^n D_\ell^n.
\end{equation}
Then the following interaction identity holds:
\begin{equation}\label{lxf_interaction}
	\begin{split}
		&\f12 \Dt \Dx \sumnN \sumj  
		\left(A_j^n E_j^n-D_j^{n} B_j^n  \right)
		\\ & \quad 
		=-\Dx \sumnN \sumjp  C_{D,j}^n \mA_j^{n+1}  
		-\Dx \sumnN \sumjp  C_{A,j}^n \f12 
		\left(\mD_{j-1}^n+\mD_{j+1}^n \right)
		\\ & \quad \qquad +I^{N+1}- I^0 + \mE, 
	\end{split}
\end{equation}
where
\begin{equation*}
\begin{split}
\mE&=
{\lambda  \over 4} \Dt \Dx \sumnN \sumjp 
E_{j-1}^n \left(B_{j+1}^n -B_{j-1}^n \right)
\\ & \qquad - {1\over 4} \Dt  \Dx  \sumnN \sumjp 
B_{j+1}^n \left(D_{j+1}^n - D_{j-1}^n\right)  
\\ & \qquad 
- {1 \over 4\lambda} \Dt  \Dx  \sumnN \sumjp 
\left( D_{j-1}^n + \lambda  E_{j-1}^n \right)
\left(A_{j+1}^n - A_{j-1}^n \right)
\\ & \qquad 
- \f12 \Dt  \Dx  \sumnN \sumjp E_{j-1}^n C_{A,j}^n
\\ &=: \mR_1+\mR_2+\mR_3+\mR_4.
\end{split}
\end{equation*}

\end{lemma}

\begin{proof}
Using the identity 
\begin{equation*}
I^n = \f12 \Dx^2 \dsum_{\Omega_n':j \le \ell} 
\left(A^n_{j+1} D^n_{\ell+1} + A^n_{j-1} D^n_{\ell-1} \right),
\end{equation*}
in combination with \eqref{def_In}, yields
\begin{equation}\label{In_diff_1}
I^{n+1}-I^n
= \Dx^2 \dsum_{\Omega'_n: j \le \ell}\left(A_j^{n+1} D_\ell^{n+1} -\f12 
\left(A^n_{j+1} D^n_{\ell+1} + A^n_{j-1} D^n_{\ell-1} \right) \right).
\end{equation}
Next,
\begin{equation}\label{int_ident_1}
\begin{split}
&A_j^{n+1} D_\ell^{n+1} -\f12 
\left(A^n_{j+1} D^n_{\ell+1} + A^n_{j-1} D^n_{\ell-1} \right)\\
&\quad =  A_j^{n+1} \left(D_{\ell}^{n+1} 
- \f12 (D_{\ell+1}^n + D_{\ell-1}^n) \right)\\
&\quad \qquad
-\f12 D_{\ell+1}^n \left(A_{j+1}^n - A_j^{n+1} \right)
- \f12 D_{\ell-1}^n \left(A_{j-1}^n - A_j^{n+1} \right)\\
&\quad =  A_j^{n+1} \left(C_{D,\ell}^n 
- {\lambda \over 2} (E_{\ell+1}^n - E_{\ell-1}^n) \right)\\
&\quad \qquad 
-\f12 D_{\ell+1}^n \left(A_{j+1}^n - A_j^{n+1} \right)
- \f12 D_{\ell-1}^n \left(A_{j-1}^n - A_j^{n+1} \right).
\end{split}
\end{equation}
Substituting the identity
\begin{equation*}
\begin{split}
&-\f12 D_{\ell+1}^n \left(A_{j+1}^n - A_j^{n+1} \right)
-\f12 D_{\ell-1}^n \left(A_{j-1}^n - A_j^{n+1} \right)\\
&\qquad =\f12 \left(D_{\ell-1}^n + D_{\ell+1}^n\right)
\left(C_{A,j}^{n} - {\lambda \over 2}(B_{j+1}^n  - B_{j-1}^n) \right)
\\ & \qquad \qquad 
-{1\over 4} \left(A_{j+1}^n - A_{j-1}^n\right) 
\left(D_{\ell+1}^n - D_{\ell-1}^n \right)
 \end{split}
\end{equation*}
into \eqref{int_ident_1} yields
\begin{equation}\label{int_ident_2}
\begin{split}
&A_j^{n+1} D_\ell^{n+1} -\f12 
\left(A^n_{j+1} D^n_{\ell+1} - A^n_{j-1} D^n_{\ell-1} \right)\\
&\quad = A_j^{n+1} \left(C_{D,\ell}^n 
- {\lambda \over 2} \left(E_{\ell+1}^n - E_{\ell-1}^n\right) \right)
\\ & \quad \qquad
+ \f12 \left(D_{\ell-1}^n + D_{\ell+1}^n\right)\left(C_{A,j}^{n} 
- {\lambda \over 2}(B_{j+1}^n  - B_{j-1}^n) \right)
\\ & \quad \qquad
- {1\over 4} \left(A_{j+1}^n - A_{j-1}^n\right) 
\left(D_{\ell+1}^n - D_{\ell-1}^n \right).
\end{split}
\end{equation}
After substituting \eqref{int_ident_2} into \eqref{In_diff_1} the result is
\begin{equation}\label{int_ident_3}
\begin{split}
I^{n+1}-I^n
& = \Dx^2 \dsum_{\Omega'_n: j \le \ell} A_j^{n+1} 
\left(C_{D,\ell}^n - {\lambda \over 2} (E_{\ell+1}^n - E_{\ell-1}^n) \right)
\\ &\qquad
+ \Dx^2 \dsum_{\Omega'_n: j \le \ell} 
\f12\left(D_{\ell-1}^n+D_{\ell+1}^n \right)
\left(C_{A,j}^{n} - {\lambda \over 2}(B_{j+1}^n  - B_{j-1}^n) \right)
\\ & \qquad 
- {1\over 4} \Dx^2 \dsum_{\Omega'_n: j \le \ell} 
\left(A_{j+1}^n - A_{j-1}^n \right) 
\left(D_{\ell+1}^n - D_{\ell-1}^n \right).
\end{split}
\end{equation}
Due to the identities
\begin{equation*}
\begin{split}
&\dsum_{\Omega'_n: j \le \ell} A_j^{n+1} 
\left(E_{\ell+1}^n - E_{\ell-1}^n\right) 
=- \sumjp A_j^{n+1} E_{j-1}^n,
\\ & 
\dsum_{\Omega'_n: j \le \ell} \f12 
\left(D_{\ell-1}^n+D_{\ell+1}^n  \right) 
\left(B_{j+1}^n - B_{j-1}^n \right) 
= \sumjp \f12 \left(D_{j-1}^n + D_{j+1}^n \right) B_{j+1}^n,
\\ &
\dsum_{\Omega'_n: j \le \ell} \left(A_{j+1}^n - A_{j-1}^n \right) 
\left(D_{\ell+1}^n-D_{\ell-1}^n \right) 
=-\sumjp \left(A_{j+1}^n - A_{j-1}^n \right) D_{j-1}^n,
\end{split}
\end{equation*}
and the identities \eqref{AD_identities}, we can 
express \eqref{int_ident_3} in the form
\begin{equation}\label{int_ident_4}
\begin{split}
I^{n+1}-I^n
& = \Dx  \sumjp \mA_j^{n+1} C_{D,j}^n 
+ {\lambda \over 2} \Dx^2 \sumjp A_j^{n+1} E_{j-1}^n 
\\ & \qquad
+\Dx  \sumjp \f12 \left(\mD_{j-1}^n+\mD_{j+1}^n \right) C_{A,j}^{n} 
\\ & \qquad
-{\lambda \over 2} \Dx^2 \sumjp \f12 
\left(D_{j-1}^n + D_{j+1}^n \right)B_{j+1}^n
\\ & \qquad 
+{1\over 4} \Dx^2 \sumjp 
\left(A_{j+1}^n - A_{j-1}^n\right)D_{j-1}^n.
\end{split}
\end{equation}
Next we use the identities
\begin{equation}\label{AE_sum}
\begin{split}
\sumjp A_j^{n+1} E_{j-1}^n 
&=  \sumjp A_{j-1}^n E_{j-1}^n + \sumjp E_{j-1}^n 
\left(A_j^{n+1} - A_{j-1}^n \right)\\
&= \sumj A_j^n E_j^n + \sumjp E_{j-1}^n 
\left(A_j^{n+1} - A_{j-1}^n \right),
\end{split}
\end{equation}
\begin{equation}\label{DB_sum}
\begin{split}
\sumjp \f12 \left(D_{j-1}^n+D_{j+1}^n \right) B_{j+1}^n
&= \sumj D_j^n B_j^n - \f12 \sumjp B_{j+1}^n 
\left(D_{j+1}^n - D_{j-1}^n \right),
\end{split}
\end{equation}
and 
\begin{equation}\label{AE_sum_1}
\begin{split}
\sumjp E_{j-1}^n \left(A_j^{n+1} - A_{j-1}^n \right)
&= \f12 \sumjp E_{j-1}^n \left(A_{j+1}^n-A_{j-1}^n \right)
\\ & \quad
-{\lambda \over 2} \sumjp E_{j-1}^n 
\left(B_{j+1}^n-B_{j-1}^n \right)
+\sumjp E_{j-1}^n C_{A,j}^n.
\end{split}
\end{equation}

Substituting \eqref{AE_sum}, \eqref{DB_sum}, 
and \eqref{AE_sum_1} into \eqref{int_ident_4}, we obtain
\begin{equation*}
\begin{split}
I^{n+1} - I^n 
& = {\lambda \over 2}  \Dx^2 \sumj  
\left(A_j^n E_j^n   - D_j^{n} B_j^n  \right)
\\ & \qquad 
+  {\lambda \over 4}  \Dx^2 \sumjp B_{j+1}^n 
\left(D_{j+1}^n - D_{j-1}^n \right)
\\ & \qquad 
- {\lambda^2  \over 4} \Dx^2 \sumjp E_{j-1}^n 
\left(B_{j+1}^n - B_{j-1}^n \right)
\\ & \qquad
+ \f12 \Dx^2 \sumjp \left(\f12 D_{j-1}^{n}
+{\lambda \over 2}E_{j-1}^{n} \right)
\left( A_{j+1}^n -A_{j-1}^n\right)
\\ & \qquad 
+ \Dx  \sumjp  C_{D,j}^n \mA_j^{n+1}
+ \Dx  \sumjp  C_{A,j}^n \f12 \left(\mD_{j-1}^n + \mD_{j+1}^n \right)
\\ & \qquad 
+ {\lambda \over 2} \Dx^2 \sumjp E_{j-1}^n C_{A,j}^n.
\end{split}
\end{equation*}
The proof is completed by summing over $n$, recalling $\lambda = \Dt  / \Dx $, 
and then solving for the sum containing $A_j^n E_j^n   - D_j^{n} B_j^n$.
\end{proof}

\begin{lemma}\label{lemma_mA,mD}
Assuming that $h = 2 \nu \Dx $, with $\nu$ a nonnegative 
integer, the following estimates hold:
\begin{equation}\label{est_mA_mD}
	\abs{\mA_{\ell}^n}, \abs{\mD_j^n} \lesssim h.
\end{equation}
\end{lemma}

\begin{proof}
We prove that $\abs{\mA_{\ell}^n} \lesssim h$.
The proof that $\abs{\mD_j^n} \lesssim h$ is similar.
Recalling \eqref{def_mA_mD},
\begin{equation*}
\begin{split}
\mA_\ell^n 
&= \Dx  \sum_{\Omega_n:j\le \ell} \chi_j \D_h U_j^n
= \Dx  \sum_{\Omega_n:j\le \ell} \chi_j 
\left( U_{j+2 \nu}^n - U_j^n\right)
\\ & =  \Dx  \sum_{\Omega_n:j\le \ell}  
\left( \chi_{j+2\nu} U_{j+2 \nu}^n - \chi_j U_j^n\right)
-\Dx  \sum_{\Omega_n:j\le \ell} U_{j + 2 \nu} \left(\chi_{j+2\nu} - \chi_j \right)
\\ & = \Dx  \sum_{i = 0}^{\nu -1}  \chi_{\ell+2 i} U_{\ell + 2i}^n
-  \Dx  \sum_{\Omega_n:j\le \ell} U_{j + 2 \nu} 
\left(\chi_{j+2\nu} - \chi_j \right).
\end{split}
\end{equation*}
This yields via \eqref{weight_properties}
\begin{equation*}
\begin{split}
\abs{\mA_\ell^n } 
&\lesssim \Dx  \sum_{i = 0}^{\nu -1}  
\chi_{\ell+2 i} + \Dx  \sum_{j \in \Omega_n} 2\nu \Dx  \chi_j
\lesssim h.
\end{split}
\end{equation*}
\end{proof}

\begin{lemma}\label{lemma_DhU_est}
Assume that $h = 2 \nu \Dx $, with $\nu$ a nonnegative 
integer. Then
\begin{equation}\label{Dh_u_Est}
	\Dt  \Dx  \sumnN \sumj \chi_j \abs{\D_h 
	\left(U_{j+2}^n - U_j^n\right) }\lesssim \sqrt{\Dx},
\end{equation}
and 
\begin{equation}\label{Dh_G_Est}
	\Dt  \Dx  \sumnN \sumj \chi_j \abs{\D_h 
	\left(G_{j+2}^n - G_j^n\right) } \lesssim \sqrt{\Dx},
	\quad \text{for $G=f, S, Q$}.
\end{equation}
\end{lemma}

\begin{proof}
For the proof of \eqref{Dh_u_Est},
\begin{equation}\label{Dh_u_Est_1}
\begin{split}
\chi_j \abs{\D_h \left(U_{j+2}^n - U_j^n\right) }
&= \chi_j \abs{\left(U_{j+2\nu+2}^n-U_{j+2\nu}^n\right) 
- \left(U_{j+2}^n- U_j^n\right)}
\\ & \le  \chi_j \abs{U_{j+2\nu+2}^n-U_{j+2\nu}^n} 
+ \chi_j \abs{U_{j+2}^n- U_j^n}
\\ & =  \chi_{j+2\nu} \abs{U_{j+2\nu+2}^n-U_{j+2\nu}^n} 
+ \chi_j \abs{U_{j+2}^n- U_j^n}
\\ & \quad + \left(\chi_j -   \chi_{j+2\nu}\right)  
\abs{U_{j+2\nu+2}^n-U_{j+2\nu}^n}
\\ & \lesssim 
\chi_{j+2\nu} \abs{U_{j+2\nu+2}^n-U_{j+2\nu}^n} 
+ \chi_j \abs{U_{j+2}^n- U_j^n}
\\ & \quad 
+ h \chi_{j+2\nu} \abs{U_{j+2\nu+2}^n-U_{j+2\nu}^n}.
\end{split}
\end{equation}
Thus
\begin{equation*}
\begin{split}
&\Dt  \Dx  \sumnN \sumj \chi_j 
\abs{\D_h \left(U_{j+2}^n - U_j^n\right) }
\lesssim (2 + h) \Dt  \Dx  \sumnN \sumj 
\chi_j \abs{U_{j+2}^n - U_j^n}
\\ & \qquad
\le (2+h) \sqrt{\Dt } \left( \Dt \Dx \sumnN \sumj 
\chi_j \right)^{1/2} 
\\ & \qquad \qquad \times
\left( \Dx  \sumnN \sumj \chi_j 
\abs{U_{j+2}^n - U_j^n}^2 \right)^{1/2}
\lesssim \sqrt{\Dx }.
\end{split}
\end{equation*}
Here we have used the Cauchy-Schwarz 
inequality, \eqref{eq:tvsquared_weight}
of Lemma~\ref{lemma_trans_sqd},
and $\Dt  = \lambda \Dx $. 

For the proof of \eqref{Dh_G_Est}, we repeat the calculation 
in \eqref{Dh_u_Est_1},with $G_j^n$ replacing $U_j^n$, resulting in
\begin{equation*}
\chi_j \abs{\D_h \left(G_{j+2}^n - G_j^n\right) }
\le (1+h)\chi_{j+2\nu} \abs{G_{j+2\nu+2}^n-G_{j+2\nu}^n}
+\chi_j \abs{G_{j+2}^n- G_j^n}.
\end{equation*}
Recalling \eqref{D_inequalities} we find that
\begin{equation*}
\begin{split}
\chi_j \abs{\D_h \left(G_{j+2}^n - G_j^n\right) }
&\lesssim
(1+h) \chi_{j+2\nu} \left( \abs{U_{j+2\nu+2}^n-U_{j+2\nu}^n}
+\abs{k_{j+2\nu+2}^n-k_{j+2\nu}^n}  \right)\\
&\qquad
+\chi_j  \left(  \abs{U_{j+2}^n- U_j^n} +  \abs{k_{j+2}^n- k_j^n}\right).
\end{split}
\end{equation*}
Summing over $n$ and $j$, invoking \eqref{Dh_u_Est} and 
the fact that $k \in BV(\R\times\R_+)$ then 
yields \eqref{Dh_G_Est}
\end{proof}

\begin{lemma}\label{lemma_space_trans}
Let $h = 2 \nu \Dx $, where $\nu$ is a nonnegative integer.
We have the following estimates for the spatial translates:
\begin{align}
	\label{spatial_translates}
	& \Dt \Dx  \sumnN \sumjp  \chi_j^2 
	\abs{U_{j-1 + 2\nu}^n - U_{j-1}^n}^{p_f+p_{S}+2}
	\lesssim h,
	\\ &
	\label{spatial_translates_1}
	\Dt  \Dx  \sumnN \sumjp  \chi_j 
	\abs{U_{j-1 + 2 \nu}^n - U_{j-1}^n}
	\lesssim h^{\mu}, \quad \mu:= 1/\left(p_f+p_{S}+2\right).
\end{align}
Under the assumptions of Remark~\ref{remark_eta_eq_f}, we obtain
\begin{equation}\label{spatial_translates_2}
	\Dt  \Dx \sumnN \sumjp  
	\chi_j \abs{U_{j-1 + 2 \nu}^n 
	- U_{j-1}^n}\lesssim h^{1/4}.
\end{equation}
\end{lemma}

\begin{proof}
We start with the proof of \eqref{spatial_translates}.
Recalling \eqref{def_ABCD_1}, 
\begin{equation}\label{ABCD_again}
A_j^n E_j^n  - D_j^{n} B_j^n 
= \chi_j^2 \left(\D_h U_j^n \D_h Q_j^n 
- \D_h S_j^n \D_h f_j^n \right).
\end{equation}

\noindent
Also, using \eqref{Dh_identity} yields
\begin{equation}\label{dot_1}
\begin{split}
\D_h U_j^n \D_h Q_j^n - \D_h S_j^n \D_h f_j^n
=\Gamma_j^n 
&+ \D_h^u U_j^n \D_{-h}^k Q_{j+2\nu}^n 
- \D_h^u S_j^n \D_{-h}^k f_{j+2\nu}^n
\\ &
- \D_{-h}^k S_{j+2\nu}^n \D_h^u f_j^n
-\D_{-h}^k S_{j+2\nu}^n \D_{-h}^k f_{j+2\nu}^n ,
\end{split}
\end{equation}
where $\Gamma_j^n$ is defined by \eqref{def_Gamma}. 
Substituting \eqref{dot_1} into \eqref{ABCD_again}, and 
then the result into the interaction 
identity \eqref{lxf_interaction} we find that
\begin{equation*}
\begin{split}
& \f12 \Dt  \Dx  \sumnN \sumj  \chi_j^2\Bigl(\Gamma_j^n 
+ \D_h^u U_j^n \D_{-h}^k Q_{j+2\nu}^n
-\D_h^u S_j^n \D_{-h}^k f_{j+2\nu}^n
\\ & \qquad \qquad \qquad \qquad \qquad \qquad
-\D_{-h}^k S_{j+2\nu}^n \D_h^u f_j^n 
-\D_{-h}^k S_{j+2\nu}^n \D_{-h}^k f_{j+2\nu}^n \Bigr)
\\ & \qquad 
= - \Dx  \sumnN \sumjp  C_{D,j}^n \mA_j^{n+1}
-\Dx \sumnN \sumjp  C_{A,j}^n \f12 
\left(\mD_{j-1}^n + \mD_{j+1}^n \right) +I^{N+1} - I^0 + \mE.
\end{split}
\end{equation*}

\noindent
Solving for the sum containing $\Gamma_j^n$, the result is
\begin{equation}\label{Gamma_sum}
\f12 \Dt  \Dx  \sumnN \sumj \chi_j^2 \Gamma_j^n
= \sum_{i=1}^4 \mS_i + I^{N+1} - I^0 + \mE ,
\end{equation}
where
\begin{equation*}
\begin{split}
& \mS_1=-\Dt  \Dx  \sumnN \sumj \f12 
\chi_j^2  \D_h U_j^n \D_{-h}^k Q_{j+2\nu}^n
\\ & \qquad \qquad 
+\Dt  \Dx  \sumnN \sumj \f12 \chi_j^2  
\D_h^u S_j^n \D_{-h}^k f_{j+2\nu}^n 
=: -\mS_{1,1} + \mS_{1,2},
\\ & 
\mS_2 = \Dt  \Dx  \sumnN \sumj \f12 
\chi_j^2   \D_{-h}^k S_{j+2\nu}^n \D_h^u f_{j}^n
\\ & \qquad \qquad
+\Dt  \Dx  \sumnN \sumj \f12 \chi_j^2
\D_{-h}^k S_{j+2\nu}^n \D_{-h}^k f_{j+2\nu}^n,
\\ & 
\mS_3 =  -\Dx \sumnN  \sumjp C_{A,j}^n \f12 
\left(\mD_{j-1}^n+\mD_{j+1}^n \right),
\\ &
\mS_4 =  -\Dx  \sumnN \sumjp C_{D,j}^n \mA_j^{n+1}.
\end{split}
\end{equation*}

\noindent
By combining \eqref{Gamma_sum} 
with \eqref{Gamma_estimate} we obtain
\begin{equation*}
\Dt  \Dx  \sumnN \sumjp  \chi_j^2 
\abs{U_{j-1 + 2\nu}^n - U_{j-1}^n}^{p_f+p_{S}+2}
\le {2\over C_{f,S}} \left(I^{N+1} - I^0 + \sum_{i=1}^4 \mS_i + \mE \right).
\end{equation*}

Next we estimate the terms on the right side of \eqref{Gamma_sum}. 
The goal is to show that each term $\lesssim h$, 
which suffices to prove \eqref{spatial_translates}. 

\medskip \noindent
{\bf Estimates of $\mS_1$, $\mS_2$.}
Recall that $\mS_1 =-\mS_{1,1}+\mS_{1,2}$.  
For  $\mS_{1,1}$ we use
\begin{equation*}
\chi^2 \lesssim 1, \quad 
\abs{\D_h^u U_j^n} \lesssim 1, \quad
\abs{\D_{-h}^k Q_{j+2}^n} \le \norm{\pk Q} \abs{k_{j}^n - k_{j-2 \nu}^n}
\end{equation*}
to obtain
\begin{equation*}
\abs{\mS_{1,1}} \lesssim \Dt  \Dx  \sumnN \sumj  
\abs{k_{j}^n - k_{j-2\nu}^n} \lesssim 2 \nu \Dx  = h.
\end{equation*}
In the last step we used that fact that $k \in BV(\R\times\R_+)$.
Similarly, using
\begin{equation*}
\abs{\D_h^u S_j^n} \lesssim 1, \quad
\abs{\D_{-h}^k f_{j+2\nu}^n} \le \norm{\pk f} \abs{k_{j}^n - k_{j-2\nu}^n}
\end{equation*}
yields $\abs{\mS_{1,2}} \lesssim h$.  An application of the 
triangle inequality then yields $\mS_1 \lesssim h$. 
A similar calculation results in $\mS_2 \lesssim h$.

\medskip \noindent 
{\bf Estimates of $\mS_3$, $\mS_4$.}
Referring to \eqref{def_ABCD_1} and using 
the fact that $\D_h U_j^n$ and $\D_h f_j^n$ are bounded 
independent of $\Dx$, $\abs{C_{A,j}^n} \lesssim  \chi_j \Dx$.
Combining this with \eqref{est_mA_mD} yields 
$\abs{C_{A,j}^n \f12 \left(\mD_{j-1}^n+\mD_{j+1}^n \right)} 
\lesssim h \chi_j \Dx$. The estimate 
$\abs{\mS_3} \lesssim h$ then follows readily.

The estimate of $\mS_4$ is similar, except for the 
additional term $\chi_j \D_h \Psi_j^n$ appearing in $C_{D,j}^n$.
The contribution to $\abs{\mS_4}$ of 
the additional term is
\begin{equation*}
\begin{split}
&\abs{\Dx  \sumnN \sumjp \chi_j \D_h \Psi_j^n \mA_j^{n+1}}
\overset{\eqref{est_mA_mD}}{\lesssim} 
h \Dx   \sumnN \sumjp \chi_j \abs{\D_h \Psi_j^n} 
\\ & \qquad 
\le h \Dx   \sumnN \sumjp \chi_j \abs{ \Psi_{j+2\nu}^n}
+h \Dx  \sumnN \sumjp \chi_j \abs{ \Psi_j^n} 
\\ & \qquad 
=  h \Dx\sumnN \sumjp \chi_{j+ 2 \nu} \abs{ \Psi_{j+2\nu}^n} 
+ h \Dx\sumnN \sumjp \chi_j \abs{ \Psi_j^n}
\\ & \qquad \qquad
+ h \Dx\sumnN \sumjp \left(\chi_{j+ 2 \nu}- \chi_j\right) 
\abs{ \Psi_{j+2\nu}^n}
\\ & \qquad  
\le 2h \Dx  \sumnN \sumjp \chi_j \abs{ \Psi_j^n}
+h^2 \Dx  \sumnN \sumjp \chi_j \abs{ \Psi_j^n}
\overset{\eqref{Psi_est}}{\lesssim} h.
\end{split}
\end{equation*}

\medskip \noindent
{\bf Estimates of $I^0, I^{N+1}$.}
Recalling \eqref{def_In} and \eqref{def_mA_mD}, 
we obtain $I^n = \Dx   \sumj A_j^n \mD_j^n$. 
Invoking \eqref{est_mA_mD} we obtain
\begin{equation*}
\begin{split}
\abs{I^n} &\lesssim h \Dx   \sumj \abs{A_j^n}
= h \Dx   \sumj \chi_j \abs{\D_h U_j^n}
\lesssim h.
\end{split}
\end{equation*}

\medskip
\noindent
{\bf Estimate\ of $\mE$.}
We claim that $\mE \rightarrow 0$ as $\Dx \rightarrow 0$, which
implies that $\mE \lesssim h$.
For the first sum of $\mE$, denoted $\mR_1$,
we use the fact that $E_j^n$ is bounded to obtain
\begin{equation}\label{est_mR_1_1}
\abs{\mR_1}
\lesssim \Dt  \Dx  \sumn \sumjp \abs{B_{j+1}^n-B_{j-1}^n}.
\end{equation}
Next,
\begin{equation*}
\begin{split}
\abs{B_{j+1}^n-B_{j-1}^n}
&= \abs{\chi_{j+1} \D_h f_{j+1}^n - \chi_{j-1} \D_h f_{j-1}^n}\\
&\le \chi_{j-1}\abs{\D_h \left(f_{j+1}^n - f_{j-1}^n\right)}
 + \abs{\chi_{j-1} - \chi_{j+1}} \abs {\D_h f_{j+1}^n}\\
 &\lesssim \chi_{j-1}\abs{\D_h \left(f_{j+1}^n - f_{j-1}^n\right)}
 +2 \Dx  \chi_{j+1}  \abs {\D_h f_{j+1}^n}.
\end{split}
\end{equation*}
Substituting this into \eqref{est_mR_1_1}, and using 
the fact that $\abs {\D_h f_{j+1}^n}$ is bounded, the result is
\begin{equation*}
\begin{split}
\abs{\mR_1}
&\lesssim \Dt  \Dx  \sumnN \sumjp
\chi_{j-1}\abs{\D_h \left(f_{j+1}^n - f_{j-1}^n\right)}
+ \Dt  \Dx  \sumnN \sumjp 2 \Dx  \chi_{j+1}\\
&\lesssim \sqrt{\Dx } + \Dx  \lesssim \sqrt{\Dx }.
\end{split}
\end{equation*}
Here we have used \eqref{Dh_G_Est}. 
Similar calculations give $\mR_2, \mR_3 \lesssim \sqrt{\Dx }$. 
We recall  $\abs{C_{A,j}^n} \lesssim  \chi_j \Dx $ to 
obtain $\mR_4 \lesssim \Dx $. This completes the 
proof of \eqref{spatial_translates}.

To prove \eqref{spatial_translates_1}, let $\tilde{\chi}(x)= \sqrt{\chi(x)}$. 
Since $\tilde{\chi}$ is also a weight function, the 
estimate \eqref{spatial_translates} holds with $\tilde{\chi}_j$ 
replacing $\chi_j$. An application of H\"{o}lder's inequality, 
with weighted measure $\chi_j \Dx  \Dt $, 
then yields \eqref{spatial_translates_1}. 

Recalling Remark~\ref{remark_eta_eq_f}, the inequality 
\eqref{spatial_translates_2} follows from \eqref{spatial_translates_1} 
by the particular choice $S(k,u) = f(k,u)$. 
The strong convexity assumption implies that 
$p_f=p_S=1$ and then $\mu = 1/4$.
\end{proof}

We are now in a position to prove the main result 
of this section, specifically, the spatial part of 
Theorem \ref{thm:main}.

\begin{proposition}\label{prop_space_est}
For $t \in [0,T]$ and $h>0$, the spatial 
translates satisfy the following estimate:
\begin{equation*}
	\int_0^T \int_\R \chi(x) 
	\abs{u^\D(x+h,t)-u^\D(x,t)} \dx \dt  \lesssim h^{\mu},
	\quad \mu:= 1/(p_f+p_{S}+2).
\end{equation*}
\end{proposition}

\begin{proof}
In the special case where $h = 2 \nu \Dx $, 
with $\nu$ a nonnegative integer,
\begin{equation}\label{space_est_h_A}
\begin{split}
&\int_0^T \int_\R \chi(x) \abs{u^\D(x+h,t)-u^\D(x,t)} \dx \dt  
\\ & \quad 
= 2 \Dt  \Dx  \sumnN \sumjp 
\chi_j \abs{U_{j-1+2\nu}^n-U_{j-1}^n}
\\ & \quad\qquad
+2 \Dt  \Dx  \sumnN \sumjp  
\abs{U_{j-1+2\nu}^n - U_{j-1}^n}\int_{x_{j-2}}^{x_j}
\left(\chi(x)-\chi_{j} \right)\dx.
\end{split}
\end{equation}
The first term on the right side of \eqref{space_est_h_A} 
is $\lesssim h^\mu$, according to Lemma~\ref{lemma_space_trans}.
The second term is $O(\Dx )$, and therefore also $\lesssim h^{\mu}$.

For more general $h>0$, we write $h = 2 \nu \Dx  + 2 \alpha \Dx $, 
where $\nu$ is a nonnegative integer and $\alpha \in (0,1)$. Then
\begin{equation}\label{space_est_h_B}
\begin{split}
&\int_0^T \int_\R \chi(x) \abs{u^\D(x+h,t)-u^\D(x,t)}\dx \dt 
\\ & \quad
\le \int_0^T \int_\R \chi(x) \abs{u^\D(x+2 \nu \Dx,t)-u^\D(x,t)} \dx \dt 
\\ & \quad \qquad 
+ \int_0^T \int_\R \chi(x) \abs{u^\D(x+2 \nu \Dx
+2 \alpha \Dx,t)- u^\D(x+2 \nu\Dx,t)} \dx \dt.
\end{split}
\end{equation}
The first integral on the right side of \eqref{space_est_h_B} 
is $\lesssim h^\mu$, according to the previous
part of the proof. The second integral is equal to
\begin{equation}\label{space_est_h_C}
\begin{split}
&\int_0^T \int_{\R} \chi(x- 2 \nu \D x)
\abs{u^{\D}(x+2 \alpha \D x,t)-u^{\D}(x,t)}\,dx\,dt
\\ & \quad 
=\int_0^T \int_{\R} \chi(x)
\abs{u^{\D}(x+2 \alpha \D x,t)-u^{\D}(x,t)}\,dx\,dt + O(h)
\\ & \quad 
=\Dt \sumnN \sumjp \chi_j \int_{x_{j-2}}^{x_j} 
\abs{u^\D (x + 2\alpha \Dx,t^n)-u^\D (x , t^n)} \,dx+O(h)
\\ & \quad 
= \alpha2 \Dt  \Dx\sumnN \sumjp \chi_{j} 
\abs{U^n_{j + 1}- U^n_{j -1}} + O(h). 
\end{split}
\end{equation}
In the last step we used
the fact that $u^{\D}$ is piecewise constant, and that
\begin{equation*}
u^\D (x + 2\alpha \Dx,t^n)-u^\D (x,t^n) 
=\begin{cases}
0, &\quad x \in [x_{j-2},x_j - 2\alpha \D x),\\
U^n_{j + 1}- U^n_{j -1}, &\quad x \in [x_j - 2\alpha \D x, x_j).
\end{cases}
\end{equation*}
Using the estimate from the first part of the proof, the 
sum on the right side of \eqref{space_est_h_C}
is $\lesssim \alpha \Dx ^\mu$, and therefore 
also $\lesssim h^{\mu}$.
\end{proof}

\section{Estimate of temporal translates}
\label{sec:temporal_est}

In this section we prove estimates that are the temporal 
analog of the spatial estimates
obtained in Section~\ref{sec:spatial_est}. 
The section concludes with 
Proposition \ref{prop_time_est}, which 
effectively provides the temporal part of the translation 
estimate of Theorem \ref{thm:main}.

We prove a sequence of lemmas leading up to 
the proof of Proposition~\ref{prop_time_est}.
For now assume  that $\tau = \theta \Dt $, where $\theta$ 
is an even nonegative integer. 
Define the temporal finite difference operator
\begin{equation*}
	\Dtau Z_j^n = Z_j^{n+\theta} - Z_j^n.
\end{equation*}
Another version of the system \eqref{lxf_sys_2} 
results by multiplying \eqref{lxf_sys_1} by $\chi_j \Dtau$:
\begin{equation*}
\left\{
\begin{split}
 \tA_j^{n+1} - \f12 \left(\tA_{j-1}^n + \tA_{j+1}^n \right) 
 +{\lambda \over 2} \left(\tB_{j+1}^n - \tB_{j-1}^n \right) = \tC_{A,j}^n, 
 \quad j \in \Omega'_n,\\
 \tD_j^{n+1} - \f12 \left(\tD_{j-1}^n + \tD_{j+1}^n \right) 
 +{\lambda \over 2} \left(\tE_{j+1}^n - \tE_{j-1}^n \right) = \tC_{D,j}^n, 
 \quad j \in \Omega'_n,
\end{split}
\right.
\end{equation*}
with \eqref{def_ABCD_1} replaced by
\begin{equation}\label{def_ABCD_2}
\begin{split}
&\tA_j^n = \chi_j \D^\tau U_j^n, \quad \tB_j^n = \chi_j \D^\tau f_j^n,
\\ & 
\tD_j^n = \chi_j \D^\tau S_j^n, \quad \tE_j^n = \chi_j \D^\tau Q_j^n,
\\  &
\tC_{A,j}^n =  - \f12 \D^\tau U_{j+1}^n \left(\chi_{j+1}-\chi_j \right)
+ \f12 \D^\tau U_{j-1}^n \left(\chi_{j} - \chi_{j-1} \right)
\\ & \qquad \qquad 
+ {\lambda \over 2} \D^\tau f_{j+1}^n \left(\chi_{j+1} - \chi_j \right)
+ {\lambda \over 2} \D^\tau f_{j-1}^n \left(\chi_{j} - \chi_{j-1} \right),
\\  &
\tC_{D,j}^n = - \f12 \D^\tau S_{j+1}^n \left(\chi_{j+1} - \chi_j \right)
+\f12 \D^\tau S_{j-1}^n \left(\chi_{j} - \chi_{j-1} \right)\\
& \qquad \qquad 
+ {\lambda \over 2} \D^\tau Q_{j+1}^n 
\left(\chi_{j+1} - \chi_j \right)
+ {\lambda \over 2} \D^\tau Q_{j-1}^n 
\left(\chi_{j} - \chi_{j-1} \right)
\\ &\qquad \qquad 
+ \chi_j \D^\tau \Psi_j^n,               
\end{split}
\end{equation}
i.e., all instances of $\D_h$ in \eqref{def_ABCD_1} 
have been replaced by $\Dtau$.

\begin{lemma}\label{lemma_time_trans_3}
Assuming that $\theta$ is an even nonnegative integer, 
the following estimate holds:
\begin{equation}\label{time_trans_3}
	\Dt  \Dx  \sumnNt \sumj \chi_j^2 
	\abs{U_j^{n+\theta} - U_j^n}^{p_f + p_S + 2}
	\lesssim \Dt  \Dx  \sumnNt \sumj  
	\left( \tA_j^n \tE_j^n - \tD_j^n \tB_j^n\right) + \tau.
\end{equation}
\end{lemma}

\begin{proof}

Define
\begin{equation*}
\begin{split}
\tGamma_j^n = 
& \left(U_j^{n+\theta} - U_j^n \right) 
\left(Q(k_j^n,U_j^{n+\theta}) - Q(k_j^n,U_j^n)\right) 
\\ & \quad
-\left(S(k_j^n,U_j^{n+\theta}) - S(k_j^n,U_j^n) \right) 
\left(f(k_j^n,U_j^{n+\theta}) - f(k_j^n,U_j^n)\right).
\end{split}
\end{equation*}
According to Lemma~\ref{lemma_entropy_ineq},
\begin{equation}\label{tGamma_1}
\tGamma_j^n \ge C_{f,S} \abs{U_j^{n+\theta}-U_j^n}^{p_f + p_S + 2}.
\end{equation}
Referring to \eqref{def_ABCD_2}, we obtain
\begin{equation}\label{tGamma_2}
\begin{split}
\tA_j^n \tE_j^n - \tD_j^n \tB_j^n
&= \chi_j^2 \left(\Dtau U_j^n \Dtau Q_j^n-\Dtau S_j^n \Dtau f_j^n \right)
= \chi_j^2 \left(\tGamma_j^n+O(\abs{k_{j}^{n+\theta} - k_{j}^n}) \right).
\end{split}
\end{equation}
Combining \eqref{tGamma_1} and \eqref{tGamma_2}, the result is
\begin{equation}\label{time_trans_1}
\chi_j^2 \abs{U_j^{n+\theta} - U_j^n}^{p_f + p_S + 2}
\lesssim \tA_j^n \tE_j^n - \tD_j^n \tB_j^n
+\chi_j^2 \abs{k_{j}^{n+\theta} - k_{j}^n}.
\end{equation}

\noindent
Summing \eqref{time_trans_1} over $n$ and $j$ yields
\begin{equation}\label{time_trans_3A}
\begin{split}
&\Dt  \Dx  \sumnNt \sumj \chi_j^2 
\abs{U_j^{n+\theta} - U_j^n}^{p_f + p_S + 2}
\\ & \qquad \lesssim \Dt  \Dx  \sumnNt 
\sumj  \left( \tA_j^n \tE_j^n - \tD_j^n \tB_j^n\right)
+ \Dt  \Dx  \sumnNt \sumj 
\chi_j^2 \abs{k_{j}^{n+\theta} - k_{j}^n}.
\end{split}
\end{equation}
We obtain \eqref{time_trans_3} 
from \eqref{time_trans_3A} by 
recalling that $k \in BV(\R\times\R_+)$.
\end{proof}

Define
\begin{equation}\label{def_mA_mD_temp}
\begin{split}
\tmA_\ell^n =  \Dx  \sum_{\Omega_n:j\le \ell} \tA_j^n, 
\quad \ell \in \Omega_n,
\qquad
\tmD_j^n = \Dx  \sum_{\Omega_n:\ell\ge j} \tD_\ell^n, 
\quad j \in \Omega_n.
\end{split}
\end{equation}

\begin{lemma}\label{lemma_mA,mD_temporal}
Assuming that $\theta$ is an even nonnegative integer,
the following estimates hold:
\begin{align}
	\label{est_mA_temporal}
	& \abs{\tmA_{\ell}^n} \lesssim \tau,
	\\ \label{est_mD_temporal}
	& \abs{\tmD_{\ell}^n} \lesssim \tau 
	+ \Dx \sum_{m=0}^{\theta-1} \sum_{j \in \Omega'_{n+m}}
	\chi_j \abs{\Psi_j^{n+m}}.
\end{align}
\end{lemma}

\begin{proof}
For the proof of \eqref{est_mA_temporal},
\begin{equation*}
\begin{split}
\tmA_\ell^n 
& = \Dx  \sum_{\Omega_n:j\le \ell} \chi_j \left( U_j^{n+\theta} - U_j^n\right)
\\ & = \sum_{m=0}^{\theta/2 - 1}  
\underbrace{\D x\sum_{\Omega_n:j\le \ell} \chi_j 
\left(U_j^{n+2(m+1)} - U_j^{n+2m} \right)}_{=: \mZ^m}.
\end{split}
\end{equation*}
It suffices to show that $\mZ^m = O(\D x)$. Without loss of generality,
take $m=0$. We employ the following identity, which 
results from \eqref{LxFr_scheme}:
\begin{equation}\label{Z_ident}
\begin{split}
U_j^{n+2}-U_j^n
&= -{\lambda \over 4} \left( f_{j+2}^n - f_{j}^n \right)
+{1 \over 4}  \left( U_{j+2}^n - U_{j}^n \right)
  -{\lambda \over 4} \left( f_{j}^n - f_{j-2}^n \right)
-{1 \over 4}  \left( U_{j}^n - U_{j-2}^n \right)
\\ & \quad 
- {\lambda \over 2} \left(f_{j+1}^{n+1} - f_{j-1}^{n+1} \right).
\end{split}
\end{equation}
We claim that the contribution to $\mZ^0$ of each term 
on the right side of \eqref{Z_ident} is $O(\D x)$.
We prove the claim for the first term. 
The proof for the other four terms is similar. 
Its contribution (ignoring the factor of $-\lambda /4$) is
\begin{equation*}
\begin{split}
\Dx\sum_{\Omega_n:j\le \ell} \chi_j   \left( f_{j+2}^n - f_{j}^n \right)
&=  \Dx\sum_{\Omega_n:j\le \ell}
\left(\chi_{j+2} f_{j+2}^n - \chi_j f_{j}^n \right) + O(\D x)\\
&= \Dx \chi_{\ell+2} f_{\ell+2}^n + O(\D x) = O(\D x).
\end{split}
\end{equation*}

For the proof of \eqref{est_mD_temporal}, we proceed 
as in the proof of \eqref{est_mA_temporal}, which yields
\begin{equation*}
\begin{split}
\tmD_{\ell}^n 
& = \Dx  \sum_{\Omega_n:j \ge \ell} 
\chi_j \left( S_j^{n+\theta} - S_j^n\right)
= \sum_{m=0}^{\theta/2 - 1}  
\underbrace{\D x\sum_{\Omega_n:j \ge \ell} \chi_j 
\left(S_j^{n+2(m+1)} - S_j^{n+2m} \right)}_{=: \mY^m}.
\end{split}
\end{equation*} 
We use the following identity, which follows from
the second equation of \eqref{lxf_sys_1}:
\begin{equation}\label{Y_ident}
\begin{split}
S_j^{n+2}-S_j^n
&= -{\lambda \over 4} \left( Q_{j+2}^n - Q_{j}^n \right)
+{1 \over 4}  \left( S_{j+2}^n - S_{j}^n \right)
-{\lambda \over 4} \left( Q_{j}^n - Q_{j-2}^n \right)
-{1 \over 4}  \left( S_{j}^n - S_{j-2}^n \right)\\
&\quad 
- {\lambda \over 2} \left(Q_{j+1}^{n+1} - Q_{j-1}^{n+1} \right)
+ \Psi_j^{n+1} + \f12 \Psi_{j+1}^n + \f12 \Psi_{j-1}^n.
\end{split}
\end{equation}
The contribution to $\mY^0$ of the first five terms on the 
right side of \eqref{Y_ident} is $O(\D x)$, as in the 
proof of \eqref{est_mA_temporal}. Thus
\begin{equation*}
\begin{split}
\abs{\tmD_{\ell}^n} 
&\lesssim \tau + \sum_{m=0}^{\theta/2 - 1}  
\D x \sum_{\Omega_n:j \ge \ell} \chi_j \left( \abs{\Psi_j^{n+2m+1}} 
+ \f12 \abs{\Psi_{j+1}^{n+2m}} + \f12 \abs{\Psi_{j-1}^{n+2m}} \right)\\
&\le \tau + \sum_{m=0}^{\theta/2 - 1}  
\D x \sum_{j \in \Omega_n} \chi_j \left( \abs{\Psi_j^{n+2m+1}} 
+ \f12 \abs{\Psi_{j+1}^{n+2m}} + \f12 \abs{\Psi_{j-1}^{n+2m}} \right)\\
&= \tau + \Dx \sum_{m=0}^{\theta-1} \sum_{j \in \Omega'_{n+m}}
\chi_j \abs{\Psi_j^{n+m}}.
\end{split}
\end{equation*}
\end{proof}

The following lemma states another interaction identity, 
analogous to \eqref{lxf_interaction}.
After the operator $\Dtau$ is substituted for $\Delta_h$, the proof is 
identical to the proof of Lemma~\ref{lemma_lxf_interaction}.

\begin{lemma}[Lax-Friedrichs interaction identity, temporal]
\label{lemma_lxf_interaction_temp}
Define
\begin{equation}\label{def_In_temp}
	\tI^n = \Dx^2 \dsum_{\Omega_n: j \le \ell} \tA_j^n \tD_\ell^n.
\end{equation}
We have the following interaction identity:
\begin{equation}\label{lxf_interaction_temp_1}
	\begin{split}
		& \f12 \Dt  \Dx \sumnNt \sumj  
		\left(\tA_j^n \tE_j^n-\tD_j^{n} \tB_j^n \right)
		\\ & \quad 
		= - \underbrace{\Dx  \sumnNt \sumjp  \tC_{A,j}^n \f12 
		\left(\tmD_{j-1}^n + \tmD_{j+1}^n \right)}_{\tmS_3}
		- \underbrace{\Dx  \sumnNt \sumjp  \tC_{D,j}^n \tmA_j^{n+1}}_{\tmS_4}  
		+\tI^{N-\theta +1} - \tI^0 + \tmE, 
	\end{split}
\end{equation}
where
\begin{equation}\label{lxf_interaction_temp_2}
	\begin{split}
		\tmE&=
		{\lambda  \over 4} \Dt  \Dx  \sumnNt \sumjp \tE_{j-1}^n 
		\left(\tB_{j+1}^n -\tB_{j-1}^n \right)
		-{1\over 4} \Dt  \Dx  \sumnNt \sumjp \tB_{j+1}^n 
		\left(\tD_{j+1}^n - \tD_{j-1}^n \right)  
		\\ & \quad 
		- {1 \over 4 \lambda} \Dt  \Dx  \sumnNt \sumjp 
		\left(\tD_{j-1}^n + \lambda \tE_{j-1}^n \right)
		\left(\tA_{j+1}^n-\tA_{j-1}^n \right)
		- \f12 \Dt  \Dx  \sumnNt \sumjp \tE_{j-1}^n \tC_{A,j}^n
		\\ & 
		=: \tmR_1 + \tmR_2 + \tmR_3 + \tmR_4.
	\end{split}
\end{equation}
\end{lemma}

Prior to utilizing this interaction identity, it is necessary 
for us to obtain some estimates, which serve as 
consequences of \eqref{eq:tvsquared_weight}.

\begin{lemma}\label{lemma_D_time}
Assuming that $\theta$ is an even nonnegative integer,
the following estimates hold:
\begin{align}\label{D_time_est_u}
	& \Dt  \Dx  \sumnNt \sumj \chi_j \abs{\Dtau 
	\left(U_{j+2}^n - U_j^n \right)} 
	\lesssim \sqrt{\Dt },
	\\ \label{D_time_est_G}
	& \Dt  \Dx  \sumnNt \sumj \chi_j 
	\abs{\Dtau \left(G_{j+2}^n - G_j^n\right) } 
	\lesssim \sqrt{\Dt },
	\quad \text{for $G=f, S, Q$}.
\end{align}
\end{lemma}

\begin{proof}
For the proof of \eqref{D_time_est_u}, we use
\begin{equation*}
\chi_j \abs{\Dtau \left(U_{j+2}^n - U_j^n \right)}
 \le \chi_{j} \abs{U_{j+2}^{n+\theta} - U_{j}^{n+\theta}}
 +\chi_j \abs{U_{j+2}^{n}-U_j^n},
\end{equation*}
which yields
\begin{equation}\label{time_diff_0}
\Dt  \Dx  \sumnNt \sumj \chi_j \abs{\Dtau \left(U_{j+2}^n - U_j^n \right)} 
\le 2 \Dt  \Dx  \sumnN \sumj  \chi_j \abs{U_{j+2}^n-U_j^n} .
\end{equation}
Applying the Cauchy-Schwarz inequality to the right side 
of \eqref{time_diff_0}, we obtain
\begin{equation*}
\begin{split}
&\Dt  \Dx  \sumnNt \sumj \chi_j \abs{\Dtau \left(U_{j+2}^n - U_j^n \right)} 
\\ & \qquad 
\le 2 \sqrt{\Dt } \left(\Dt  \Dx  \sumnN \sumj \chi_j  \right)^{1/2} 
\left(\Dx  \sumnN \sumj \chi_j \abs{U_{j+2}^n-U_j^n}^2 \right)^{1/2}. 
\end{split}
\end{equation*}
Recalling \eqref{eq:tvsquared_weight} we have \eqref{D_time_est_u}.

For the proof of \eqref{D_time_est_G}, we use
\begin{equation*}
\begin{split}
\chi_j \abs{\Dtau \left(G_{j+2}^n - G_j^n \right)}
&\le \chi_j \abs{G_{j+2}^{n+\theta} - G_j^{n+\theta}} 
+ \chi_j \abs{G_{j+2}^n-G_j^n}
\\ & 
\le \chi_j \norm{\partial_u G} 
\abs{U_{j+2}^{n+\theta} - U_{j}^{n+\theta}}
+ \chi_j \norm{\partial_k G} 
\abs{k_{j+2}^{n+\theta} - k_{j}^{n+\theta}}
\\ & \qquad 
+ \chi_j \norm{\partial_u G} \abs{U_{j+2}^n - U_{j}^n} 
+\chi_j \norm{\partial_k G} \abs{k_{j+2}^n - k_{j}^n}.    
\end{split}
\end{equation*}
Summing over $n$ and $j$ then results in
\begin{equation}\label{D_time_est_pf_G_1}
\begin{split}
& \Dt  \Dx  \sumnNt \sumj \chi_j \abs{\D_h \left(G_{j+2}^n - G_j^n\right)} 
\\ & \qquad
\lesssim \Dt  \Dx  \sumnN \sumj \chi_j \abs{U_{j+2}^n - U_{j}^n} 
+\Dt  \Dx  \sumnN \sumj \chi_j \abs{k_{j+2}^n - k_{j}^n}. 
\end{split}
\end{equation}
The proof of \eqref{D_time_est_G} is completed by 
estimating each of the sums on the right side
of \eqref{D_time_est_pf_G_1}. 
The first sum is estimated using the Cauchy-Schwarz 
inequality as in the proof of \eqref{D_time_est_u}.
For the second sum we use the fact 
that $k \in BV(\R\times\R_+)$.
\end{proof}

\begin{lemma}\label{lemma_time_est_even}
Assuming that $\theta$ is an even nonnegative integer, we 
have the following estimate for the temporal translates:
\begin{equation*}
	\Dt  \Dx  \sumnNt \sumj \chi_j^2 
	\abs{U_j^{n+\theta} - U_j^n}^{p_f + p_S + 2}
	\lesssim \tau.
\end{equation*}
\end{lemma}

\begin{proof}
Recalling \eqref{time_trans_3}, it suffices to show that
\begin{equation*}
\Dt  \Dx  \sumnNt \sumj  
\left(\tA_j^n \tE_j^n-\tD_j^n \tB_j^n\right) \lesssim \tau.
\end{equation*}
In view of the interaction identity 
\eqref{lxf_interaction_temp_1}, \eqref{lxf_interaction_temp_2}, the 
proof then reduces to proving that each term on the right side 
of \eqref{lxf_interaction_temp_1} is $\lesssim \tau$.

\medskip \noindent
{\bf Estimate of $\tmS_3$.}
For the estimate of $\tmS_3$, we have via 
\eqref{def_ABCD_2} and \eqref{weight_properties} 
$\abs{\tC_{A,j}^n} \lesssim \Dx  \chi_j$. 
Combining this with \eqref{est_mD_temporal}, we obtain
\begin{equation*}
\begin{split}
& \abs{\tC_{A,j}^n \f12 \left(\tmD_{j-1}^n+\tmD_{j+1}^n \right)} 
\lesssim \Dx  \tau \chi_j 
+ \Dx  \chi_j \left(\Dx \sum_{m=0}^{\theta-1}
\, \sum_{i \in \Omega'_{n+m}}
\chi_i \abs{\Psi_i^{n+m}} \right).
\end{split}
\end{equation*}
Thus
\begin{equation*}
\begin{split}
\abs{\tmS_3}
& \lesssim
\Dx  \sumnNt \sumjp \Dx  \tau \chi_j
+ \Dx  \sumnNt \sumjp \Dx  \chi_j 
\left(\Dx \sum_{m=0}^{\theta-1}\, \sum_{i \in \Omega'_{n+m}}
\chi_i \abs{\Psi_i^{n+m}} \right)\\
& =: \Sigma_1 + \Sigma_2.
\end{split}
\end{equation*}
Clearly $\Sigma_1 \lesssim \tau$. It remains to show 
that also $\Sigma_2 \lesssim \tau$. Let
\begin{equation}\label{def_Yn}
Y^n = \Dx  \sumjp \chi_j \abs{\Psi_j^n}.
\end{equation}
Then
\begin{equation*}
\begin{split}
\Sigma_2 &= \Dx  \sumnNt \sumjp \Dx  \chi_j 
\left(\sum_{m=0}^{\theta-1}\, Y^{n+m} \right)
= \left( \sumjp \Dx  \chi_j  \right)
 \Dx  \sumnNt \sum_{m=0}^{\theta-1}\, Y^{n+m}
\\ & 
\lesssim \Dx  \sumnNt \sum_{m=0}^{\theta-1}\, Y^{n+m}
= \Dx  \sum_{m=0}^{\theta-1}\, \sumnNt \, Y^{n+m}
\le \Dx  \sum_{m=0}^{\theta-1}\, \sumnN \, Y^n
= \tau \sumnN \, Y^n.
\end{split}
\end{equation*}
The estimate $\Sigma_2 \lesssim \tau$ then follows 
from \eqref{Psi_est} and
\eqref{def_Yn}.

\medskip \noindent
{\bf Estimate of $\tmS_4$.}
Referring to \eqref{def_ABCD_2} and \eqref{est_mA_temporal},
\begin{equation*}
\abs{\tC_{D,j}^n} \lesssim \Dx  \chi_j 
+ \chi_j \abs{\Dtau \Psi_j^n}, 
\quad
\abs{\tmA_j^n } \lesssim \tau.
\end{equation*}
which yields
\begin{equation*}
\begin{split}
\abs{\tmS_4}
& \lesssim 
\Dx  \sumnN \sumjp \tau \Dx  \chi_j 
+ \Dx  \sumnN \sumjp \tau \chi_j \abs{\Dtau \Psi_j^n}
\\ & 
\lesssim \tau + \tau \Dx  \sumnN \sumjp \chi_j \abs{\Psi_j^n}
\lesssim \tau.
 \end{split}
\end{equation*}
Here we have used \eqref{Psi_est} and
\begin{equation*}
\Dx  \sumnN \sumjp  \chi_j \abs{\Dtau \Psi_j^n} \le 
2 \Dx  \sumnN \sumjp  \chi_j \abs{\Psi_j^n},
\end{equation*}
which follows from the triangle inequality.

\medskip \noindent
{\bf Estimates of $\tI^0, \tI^{N-\theta +1}$.} 
Recalling \eqref{def_In_temp} and \eqref{def_mA_mD_temp}, we have
$\tI^n = \Dx   \sum_{\ell \in \Omega_n} \tD_{\ell}^n \tmA_{\ell}^n$. 
Invoking \eqref{est_mA_temporal} we obtain
\begin{equation*}
\begin{split}
\abs{\tI^n} 
\lesssim \tau \Dx   \sum_{\ell \in \Omega_n} \abs{\tD_{\ell}^n}
= \tau \Dx   \sum_{\ell \in \Omega_n} \chi_{\ell} \abs{ \Dtau S_{\ell}^n}
\lesssim \tau.
\end{split}
\end{equation*}

\medskip \noindent
{\bf Estimate of $\tmE$.}
We start by estimating $\tmR_1$, which we write as
\begin{equation*}
\begin{split}
\tmR_1 
&= {\lambda  \over 4} \Dt  \Dx  \sumnN \sumjp \tE_{j-1}^n 
\left(\tB_{j+1}^n -\tB_{j-1}^n \right)\\
&= {\lambda  \over 4} \Dt  \Dx  \sumnN \sumj \tE_{j}^n 
\left(\tB_{j+2}^n -\tB_{j}^n \right).
\end{split}
\end{equation*}
Using the fact that $\tE_{j}^n$ is bounded, along 
with \eqref{weight_properties} 
and \eqref{def_ABCD_2} we find that
\begin{equation}\label{est_mR1}
\abs{\tmR_1} 
\lesssim \Dt  \Dx  \sumnN \sumj \chi_j 
\abs{\Dtau f_{j+2}^n - \Dtau f_j^n}
\lesssim \sqrt{\Dt } + \Dt  \lesssim \sqrt{\Dt }.
\end{equation}
In the last step we used \eqref{D_time_est_G} 
of Lemma~\ref{lemma_D_time}.
Similar calculations yield
\begin{equation}\label{est_mR23}
\abs{\tmR_2}, \abs{\tmR_3 } \lesssim \sqrt{\Dt }.
\end{equation}
To estimate $\tmR_4$, we use the fact that 
$\tmE_{j-1}^n$ is bounded,
along with the estimate $\abs{\tC_{A,j}^n} \lesssim \Dx  \chi_j$, 
from which it is clear that 
\begin{equation}\label{est_mR4}
\tmR_4 \lesssim \Dt .
\end{equation}
In view of \eqref{est_mR1}, 
\eqref{est_mR23}, \eqref{est_mR4}, we have 
$\tmE \lesssim \sqrt{\Dt }  \lesssim \tau$.
\end{proof}

\begin{lemma}\label{lemma_lxf_time_trans}
We have the following estimate for the temporal translates:
\begin{equation}\label{lxf_time_trans}
	\Dt  \sumn \int_{\R} \chi(x) 
	\abs{u^{\D}(x,t^{n+1}) - u^{\D}(x,t^n)} \dx 
	\lesssim \Dt^{\mu}, 
	\quad \mu:= 1/\left(p_f+p_{S}+2\right).
\end{equation}
\end{lemma}

\begin{proof}
We write \eqref{lxf_time_trans} in the form
\begin{equation*}
\begin{split}
& \Dt  \sumn \int_{\R} \chi(x) 
\abs{u^{\D}(x,t^{n+1}) - u^{\D}(x,t^n)} \dx 
\\ & \quad 
= \Dt  \sumn \sumjp \int_{x_{j-1}}^{x_j} 
\chi(x) \abs{U_j^{n+1} - U_{j-1}^n} \dx
\\ &\quad \qquad
+ \Dt  \sumn \sumjp \int_{x_{j}}^{x_{j+1}} 
\chi(x) \abs{U_j^{n+1} - U_{j+1}^n} \dx
=: \Sigma_1 + \Sigma_2
\end{split}
\end{equation*}
We estimate $\Sigma_1$; the estimate of $\Sigma_2$ 
is similar. We have
\begin{equation*}
\begin{split}
\abs{U_j^{n+1}-U_{j-1}^n}
&=  \abs{U_j^{n+1}- \f12 (U_{j-1}^n + U_{j+1}^n)
+\f12 (U_{j+1}^n - U^n_{j-1})}
\\ &\le   \abs{U_j^{n+1}- \f12 (U_{j-1}^n + U_{j+1}^n)} 
+ \f12 \abs{U_{j+1}^n - U^n_{j-1}}
\\ &= {\lambda \over 2} \abs{f_{j+1}^n - f_{j-1}^n}
+ \f12 \abs{U_{j+1}^n - U^n_{j-1}}
\\ &\le {\lambda \over 2} \left(\norm{\pu f} \abs{U_{j+1}^n - U^n_{j-1}}
+\norm{\pk f} \abs{k_{j+1}^n - k^n_{j-1}}\right)
+ \f12 \abs{U_{j+1}^n - U^n_{j-1}}.
\end{split}
\end{equation*}
This estimate yields
\begin{equation*}
\Sigma_1 \lesssim
\Dt  \Dx  \sumn \sumjp \chi_j  \abs{U_{j+1}^n - U^n_{j-1}}
+ \Dt  \Dx  \sumn \sumjp  \abs{k_{j+1}^n - k^n_{j-1}}.
\end{equation*}
Using $k \in BV(\R\times\R_+)$ and the 
spatial estimate \eqref{spatial_translates_1},
we obtain the desired estimate for $\Sigma_1$.
\end{proof}

We can now prove the main result of this section.

\begin{proposition}\label{prop_time_est}
Fix $\tau \in [0,T]$. Then
\begin{equation}\label{time_est_tau}
	\int_0^{T-\tau} \int_\R \chi(x) \abs{u^\D(x,t+\tau)-u^\D(x,t)} 
	\dx \dt  \lesssim \tau^{\mu}, 
	\quad \mu:= 1/\left(p_f+p_{S}+2\right).
\end{equation}
\end{proposition}

\begin{proof}
First assume that $\tau = \theta \Dt $, where $\theta$ 
is an even nonnegative integer.
The integral on the left side of \eqref{time_est_tau} is
\begin{equation*}
\int_0^{T-\tau} \int_\R \chi(x) \abs{u^\D(x,t+\tau)-u^\D(x,t)} \dx \dt  
= 2 \Dt  \Dx  \sumnNt \sumj \chi_j \abs{U_j^{n+\theta}- U_j^n}.
\end{equation*}
The estimate \eqref{time_est_tau} then follows 
directly from  Lemma~\ref{lemma_time_est_even}, along
with H\"{o}lder's ineqality as in the proof of Lemma~\ref{lemma_space_trans}.

Next suppose that $\theta$ is an odd positive integer. Then 
\begin{equation*}
\begin{split}
&\int_0^{T-\tau} \int_\R \chi(x) \abs{u^\D(x,t+\tau)-u^\D(x,t)} \dx \dt  
\\& \quad \le \int_0^{T-\tau} \int_\R \chi(x) 
\abs{u^\D(x,t+(\theta-1) \Dt)-u^\D(x,t)} \dx \dt 
\\ & \quad \qquad 
+ \int_0^{T-\tau} \int_\R \chi(x) 
\abs{u^\D(x,t + \theta  \Dt)-u^\D(x,t+ (\theta -1) \Dt )} \dx \dt.
\end{split}
\end{equation*}
The desired estimate then follows from the  result 
for $\theta$ even that was proven above,
along with Lemma~\ref{lemma_lxf_time_trans}.

Finally, take the case where $\tau$ is not an 
integral multiple of $\Dt$, say
$\tau = \theta \Dt  + \alpha \Dt $, where $\theta$ is 
a nonnegative integer and $\alpha \in (0,1)$. Then 
\begin{equation}\label{time_est_tau_3}
\begin{split}
&\int_0^{T-\tau} \int_\R \chi(x) \abs{u^\D(x,t+\tau)-u^\D(x,t)} \dx \dt  
\\ & \quad 
\le \int_0^{T-\tau} \int_\R \chi(x) \abs{u^\D(x,t+\theta \Dt )-u^\D(x,t)} \dx \dt  
\\ & \quad \qquad 
+ \int_0^{T-\tau} \int_\R \chi(x) 
\abs{u^\D(x,t+\theta \Dt+\alpha \Dt)-u^\D(x,t+\theta \Dt )} \dx \dt .
\end{split}
\end{equation}
The first integral on the right side of \eqref{time_est_tau_3} 
is $\lesssim \tau^\mu$, using the previous 
part of the proof.  For the second integral, we use the 
fact that $u^{\D}$ is piecewise constant, and that
\begin{equation*}
\begin{split}
&
u^\D(x,t+\theta \Dt+\alpha \Dt) -u^\D(x,t+\theta \Dt ) \\
&\qquad 
=\begin{cases}
0, &\quad t \in [t^n,t^{n+1} - \alpha \Dt),
\\
u^\D(x,t+\theta \Dt + \Dt) - u^\D(x,t+\theta \Dt), 
&\quad t \in [t^{n+1} - \alpha \Dt, t^{n+1}).
\end{cases}
\end{split}
\end{equation*}
Thus the second integral on the right side 
of \eqref{time_est_tau_3} is equal to
\begin{equation*}
\alpha \int_0^{T-\tau} \int_\R \chi(x) 
\abs{u^\D(x,t+\theta \Dt+\Dt)-u^\D(x,t+\theta \Dt)} \dx \dt,
\end{equation*}
which is $\lesssim \Dt ^{\mu}\lesssim \tau^\mu$ by 
Lemma~\ref{lemma_lxf_time_trans}.
\end{proof}




\end{document}